\documentclass[11pt]{article}
\usepackage{vmargin}
\usepackage{times}

\usepackage[T1]{fontenc}
\usepackage{epsfig}
\usepackage{graphicx}
\usepackage{amssymb}
\usepackage{amsmath}
\usepackage[utf8]{inputenc}
\usepackage{amsfonts}
\usepackage{amsthm}
\usepackage{paralist}
\usepackage{stmaryrd}
\usepackage[toc,page]{appendix}

\usepackage[colorlinks=true]{hyperref}

\usepackage[dvipsnames]{xcolor}

\theoremstyle{plain}
\newtheorem{corollary}{Corollary}[section]
\newtheorem{theorem}[corollary]{Theorem}
\newtheorem{lemma}[corollary]{Lemma}
\newtheorem{proposition}[corollary]{Proposition}
\newtheorem*{theorem*}{Theorem}
\newtheorem*{lemma*}{Lemma}
\newtheorem*{definition*}{Definition}
\newtheorem*{corollary*}{Corollary}
\theoremstyle{definition}
\newtheorem{algorithm}{Algorithm}

\theoremstyle{remark}

\newtheorem*{remark}{Remark}

\newcommand{\R}{\mathbb{R}}

\newcommand{\proba}{\mathbb{P}}
\newcommand{\N}{\mathbb{N}}

\newcommand{\E}{\mathbb{E}}

\newcommand{\T}{\mathcal{T}}

\newcommand{\F}{\mathbb{F}}
\newcommand{\f}{\mathcal{F}}

\newcommand{\X}{\mathcal{X}}
\newcommand{\1}{\mathbf{1}}
\newcommand{\I}{\mathcal{I}}

\newcommand{\e}{\varepsilon}

\DeclareMathOperator*{\limit}{\longrightarrow}
\DeclareMathOperator*{\asym}{\sim}

\DeclareMathOperator*{\SB}{SB}

\DeclareMathOperator*{\Varr}{Var}

\DeclareMathOperator*{\Covv}{Cov}

\DeclareMathOperator*{\diam}{diam}
\DeclareMathOperator*{\Exp}{Exp}
\DeclareMathOperator*{\st}{st}

\begin{document}
\title{Compactness and fractal dimensions of inhomogeneous continuum random trees.}
\author{Arthur Blanc-Renaudie\thanks{LPSM, Sorbonne Universit\'e, France, Email: arthur.blanc-renaudie@sorbonne-universite.fr} }
\date{\today}
\maketitle
\begin{abstract}
We introduce a new stick-breaking construction for inhomogeneous continuum random trees (ICRT). This new construction allows us to prove the necessary and sufficient condition for compactness conjectured by Aldous, Miermont and Pitman  \cite{ExcICRT} by comparison with L\'evy trees. We also compute the fractal dimensions (Minkowski, Packing, Hausdorff).
\end{abstract}

\section{Introduction}
Since the pioneer work of Aldous in \cite{Aldous1}, the study of continuum random trees (CRT) is considered as a powerful tool to study properties of large random discrete trees. In particular, it has been conjectured in \cite{Aldous1}, that the Brownian CRT is a universal limit for numerous models of trees with large height. 
This has been verified over and over. Furthermore the Brownian CRT model has been extended, for discrete trees with smaller height, toward two main distinct directions. On the one hand, L\'evy trees are introduced, in Le Gall Duquesne \cite{Duquesne1,Duquesne2}, as limits of Galton-Watson trees. On the other hand,  inhomogeneous continuum random trees (ICRT) are introduced by Aldous, Camarri and Pitman, in \cite{IntroICRT1,IntroICRT2}, as limits of $\mathcal{P}$-trees. Those two distinct but similar models leave the following main problem: Finding a universal model for limits of random discrete trees (with no restriction on the height).

To solve this problem, we prove in a forthcoming paper \cite{Uniform}, that ICRT appears as limits of uniform random trees with fixed degree sequence. Since many models of interest can be studied under the spectrum of those trees, this proves that ICRT are universal. In particular L\'evy trees are ICRT with random parameters. The aim of the present paper is twofold: obtain refined information about the ICRT, the universal limit object in particular concerning compactness and fractal dimensions, and introduce some tools for convergence that will be used in \cite{Uniform}.

Our main results are derived from a new version of the stick-breaking construction of the ICRT from Aldous, Pitman \cite{IntroICRT1}. Stick-breaking constructions generate a $\R$-tree (a loopless geodesic space see Le Gall \cite{Legall} for an extensive treatment) and are separated in two steps: 
\begin{compactitem}
\item the line $\R^+$ is first cut into the segments ("sticks") $[0, Y_1],\, (Y_1,Y_2], (Y_2,Y_3] \dots$
\item  the segments are then re-arranged sequentially in a tree-like fashion by gluing $(Y_i,Y_{i+1}]$ at a point $Z_i\leq Y_i$. (see Figure \ref{SB})
\end{compactitem}

Such a construction has been introduced by Aldous \cite{Aldous1} for the Brownian CRT. Recently Amini, Devroye, Griffiths, Olver in \cite{Amini} studied a case where cuts are fixed with $(Y_{i+1}-Y_i)_{i\in \N}$ decreasing. The condition of monotonicity has been removed by Curien and Haas in \cite{Curien} where they construct a probability measure on $\T$, give a sufficient criterion for compactness of $\T$ and compute the Hausdorff dimension of $\T$. We use similar methods in a setting where cuts and glue points are generated according to a random measure $\mu$ on $\R^+$.

\paragraph{Plan of the paper} In the next section we present the new construction for ICRT. Our main results are then stated in Section \ref{Section 3}. In Section \ref{Section 4} we study the measure $\mu$ and cuts. A probability measure is constructed from $\mu$ in Section \ref{Section 5}. The compactness and fractal dimensions are the topics of Sections \ref{Section 6} and  \ref{Section 7} respectively.

\section{Model and definition of the fractal dimensions} \label{Section 2}
\subsection{The ICRT and its construction} \label{def}
Let us first present a generic deterministic stick-breaking construction. It takes for input two sequences in $\R^+$ called cuts ${\textbf y}=(y_i)_{i\in \N}$ and glue points ${\textbf z}=(z_i)_{i\in \N}$, which satisfy
\begin{equation} \forall i<j,\ \ y_i<y_j \qquad ; \qquad y_i\limit \infty \qquad ; \qquad \forall i\in \N,\ \ z_i\leq y_i, \label{2609} \end{equation}
and creates an $\R$-tree by recursively "gluing" segment $(y_i,y_{i+1}]$ at position $z_i$ (see Figure \ref{SB}), or rigorously, by constructing recursively a consistent sequence of distances $(d_n)_{n\in \N}$ on $([0,y_n])_{n\in \N}$.
\begin{figure}[!h]  \label{SB}
\centering
\includegraphics[scale=0.6]{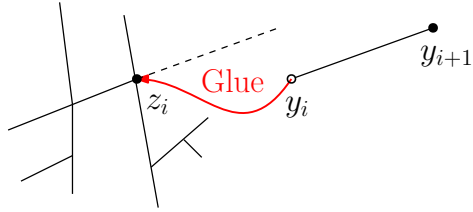}
\caption{A typical step of the stick-breaking construction: the "gluing" of $(y_i,y_{i+1}]$ at $z_i$. }
\label{SB}
\end{figure}
\begin{algorithm} \label{Alg1} \emph{Generic stick-breaking construction.}
\begin{compactitem}
\item[--] Let $d_0$ be the trivial distance on $\{0\}$.
\item[--] For each $n\geq 1$ define $d_n$ on $[0, y_n]$ such that for each $x\leq y$: 
\[ d_n(x,y):=
\begin{cases} 
d_{n-1}(x,y) & \text{if } x,y\in [0, y_{n-1}] \\
d_{n-1}(x,z_{n-1})+|y-y_{n-1}| & \text{if } x \in [0, y_{n-1}], \, y \in (y_{n-1}, y_n] \\
|x-y| &  \text{if } x,y\in (y_{n-1}, y_n]
\end{cases} \]
where by convention $y_0:=0$ and $z_0:=0$.
\item[--] Let $d$ be the unique metric on $\R^+$ which agrees with $d_n$ on $[0, y_n]$ for each $n\in \N$.
\item[--] Let $\SB({\textbf y},{\textbf z})$ be the completion of $(\R^+,d)$.
\end{compactitem}
\begin{remark}
There is a more general way of gluing metric space. (see \cite{Glue} for definition or \cite{Seni} for similar work in this context).  We prefer to work directly on $\R^+$ for practical reasons. %
\end{remark}
\end{algorithm}
We now introduce the probability space that will be used in the paper. Note that the space $\Upsilon:=\bigcup_{{\textbf y},{\textbf z}} \SB({\textbf y},{\textbf z})$ is in bijection with the space of couples of sequences $({\textbf y},{\textbf z})$ that satisfy \eqref{2609}, hence one can naturally define the weak topology on $\Upsilon$. Then we work on a complete probability space such that every random variable defined below are measurable for the weak topology.

Now, let $\Omega$ be the space of sequences $\{\theta_i\}_{i\in \N}$ in $\R^+$ such that: 
\[ \sum_{i=0}^\infty \theta_i^2=1 \quad ; \quad \theta_1\geq \theta_2 \geq \dots \quad ; \quad  \theta_0\neq 0 \text{ or } \sum_{i=1}^\infty \theta_i=\infty. \] The ICRT of parameter $\Theta\in \Omega$ is the random $\R$-tree constructed via the following algorithm.
\begin{algorithm} \label{Alg2} \emph{Classical construction of the $\Theta$-ICRT from \cite{IntroICRT1,IntroICRT2}} \begin{compactitem} 
\item[--] Let $(A_i, B_i)_{i\in \N}$ be a Poisson point process of intensity $\theta_0^2$ on $\{(a,b)\in \R^{+2}: b\leq a \}$.
\item[--] Let $((A_{i,j} )_{j\in\N})_{i\in \N}$ be a family of independent Poisson point processes of intensity $(\theta_i)_{i\in \N}$ on $\R^+$ and independent of $(A_i, B_i)_{i\in \N}$.
\item[--] Sort the elements of the (almost surely) locally finite set $\bigcup_{i=0}^\infty \{A_i\} \cup\bigcup_{i=1}^\infty \bigcup_{j=1}^\infty \{A_{i,j} \}$ as $U=(U_i)_{i\geq 1}$ with $U_1<U_2<\dots$
\item[--] For $i\geq 1$, let $V_i=  \begin{cases} Y_j \, \, \, \, \text{if } U_i \text{ is of the form} & A_j \\ A_{i,0}  \text{ -----------------------} & A_{i,j} \end{cases}$ and let $V=(V_i)_{i\geq 1}$.
\item[--] The (old) $\Theta$-ICRT is defined as $(\T^*,d^*)=\SB(U,V)$.
\end{compactitem}
\end{algorithm}
For technical reasons, it is convenient to deal with the following alternative construction.
\begin{algorithm} \label{Alg3} \emph{New construction of the $\Theta$-ICRT}
\begin{compactitem}
\item[--] Let $(X_i)_{i\in \N}$ be a family of independent exponential random variables of parameter $(\theta_i)_{i\in \N}$.
\item[--] Let $\mu$ be the measure on $\R^+$ defined by $\mu=\theta_0^2 dx+\sum_{i=1}^{\infty} \delta_{X_i} \theta_i$.
\item[--] For each $l\in\R^+$ let $\mu_l$ be the restriction of $\mu$ to $[0,l]$.
\item[--] Let $(Y_i)_{i\in \N}$ be a Poisson point process  on $\R^+$ of rate $\mu[0,l]dl$. \item[--] Let $(Z_i)_{i\in \N}$ be a family of independent random variables with respective laws $\frac{\mu_{Y_i}}{\mu[0,Y_i]}$, $i\in \N$.
\item[--] The (new) $\Theta$-ICRT is defined as $(\T,d)=\SB(Y,Z)$.
\end{compactitem}
\end{algorithm}
\begin{remark}
The construction may fail because $\mu[0,l]$ may be infinite for some $l$. However, since $\mu[0,l]$ is of finite expectation, this almost surely never happens. (See Lemma \ref{2})
\end{remark}
The constructions in Algorithms \ref{Alg2} and \ref{Alg3} are equivalent that is: 
\begin{lemma} \label{equiv}
$(\T^*,d^*)$ and $(\T,d)$ have the same distribution.
\end{lemma}
\begin{proof} First conditionally on $\{A_{i,0}\}_{i\in \N}$, $\{U_i,V_i\}_{i\in \N}$ is a Poisson point process on $\Delta:=\{(a,b)\in \R^{+2}: b\leq a \}$ of intensity 
\[ \theta_0^2dx dy+\sum_{i=1}^\infty \theta_i \1_{A_{i,0}\leq x} dx \times \delta_{A_{i,0}}. \]
Also, conditionally on $(X_i)_{i\in \N}$, $(Y_i,Z_i)_{i\in\N}$  is a Poisson point process on $\Delta$ of intensity 
\[ \theta_0^2dx dy+\sum_{i=1}^\infty \theta_i \1_{X_i\leq x} dx \times \delta_{X_i}. \]
So since $(X_i)_{i\in \N}$ and $\{A_{i,0}\}_{i\in \N}$ have the same distribution, $(U,V)$ and $(Y,Z)$ also have the same distribution. Finally $(\T^*,d^*)=\SB(U,V)$ and $(\T,d)=\SB(Y,Z)$ have the same distribution.
\end{proof}
Finally let us introduce some notation that will simplify many expressions later. 
\begin{definition*}
For $n\in \N$ let $l_n:=Y_{n}-Y_{n-1}$ denotes the length of the $n$th segment, and let $m_{n}:=\mu(Y_{n-1},Y_{n}]$ denote its weight. Then let $M_n:=\mu[0,Y_n]=m_1+\dots+m_n$. 
\end{definition*}
\subsection{Fractal dimension} \label{2.5}
In the entire section $X$ is a metric space and for every $x\in X$, $\e>0$, $B(x,\e)$ denotes the closed ball centered at $x$ with radius $\e$. We recall the definitions of the fractal dimensions we compute in this paper.
\begin{definition*} (Minkowski dimensions) For every $\e>0$ let $N_\e$ be the minimal number of closed balls of radius $\e$ to cover $X$. Define the Minkowski lower box and upper box dimensions respectively by
\[ \underline{\dim}(X):=\liminf_{l\to \infty} \frac{ \log N_{1/l}}{\log l} \quad \text{and} \quad \overline{\dim}(X):= \limsup_{l\to \infty} \frac{ \log N_{1/l}}{\log l}. \]
\end{definition*}
\begin{definition*} (Packing dimension) For every $s\geq 0$ and $A\subset X$ let 
\[ P^s_0(A):= \limsup_{\delta\to \infty} \left \{ \sum_{i\in I} \diam(B_i)^s \Bigg | \, \{B_i\}_{i\in I} \text{ are disjoint balls $B(x,r)$ with $x\in A$ and $r\leq \delta$}\right \}. \]
and
\[ P^s(X):=\inf \left \{\sum_{i=1}^{\infty} P^s_0(A_i) \Bigg | X\subset \bigcup_{i=1}^{\infty} A_i \right \}. \]
Then $P^s$ is a decreasing function of $s$, and we define the packing dimension of $X$ as 
\[ \dim_P(X):= \sup \{s, P^s(X)<\infty\}. \]
\end{definition*}
\begin{definition*} (Hausdorff dimension) For every $s,r\geq 0$ write
\[ H^s_r(X):= \inf_{\diam(A_i)\leq r} \left \{ \sum_{i=1}^{\infty} \diam(A_i)^s \Bigg | X \subseteq \bigcup_{i=1}^{\infty} A_i \right \}. \]
The Hausdorff dimension of $X$ is defined by
\[ \dim_H(X):=\sup \left \{s, \sup_{r\in \R^+}H_r^s(X)<\infty \right \}. \]
\end{definition*}
To compute the Packing dimension and Hausdorff dimension of the ICRT we will use the following extension of Theorem 6.9, and Theorem 6.11 from \cite{fractal}. (\cite{fractal} deals with subsets of Euclidian space, but the same arguments hold for every metric space.) 
\begin{lemma}\label{Hausdorff} Let $p$ be a Borel probability measure on $X$ and $s\in \R^+$.
 \begin{compactitem} 
  \item[a)]If $p$-almost everywhere $\liminf p(B(x,\e))\e^{-s}<+\infty$ as $\e\to 0$, then $\dim_P(X)\geq s$.
 \item[b)]If $p$-almost everywhere $p(B(x,\e))=O(\e^{s})$ as $\e\to 0$, then $\dim_H(X)\geq s$.
 \end{compactitem}
\end{lemma}
We have the well-known inequalities (see e.g. Chapter 3 of Falconer \cite{FalconPunch}):
\begin{lemma} \label{FalconPunch} For every metric space $X$ we have
\[ \dim_H(X)\leq  \underline{\dim}(X) \leq \overline{\dim}(X) \quad \text{and} \quad \dim_H(X)\leq  \dim_P(X) \leq \overline{\dim}(X). \]
\end{lemma}
\section{Main results}  \label{Section 3}
The first theorem defines a probability measure on ICRT. 
\begin{theorem} \label{THM1}
Almost surely there is a probability measure $p$ on the tree $\T$ such that
\[ p_l:=\frac{\mu_l}{\mu[0,l]} \limit^{\text{weakly}}_{l\to \infty} p. \]
Furthermore $p$ has support $\T$, has no atoms and gives measure $1$ to the set of leaves (the set of $x\in \T$ such that $\T\backslash\{x\}$ is connected).
\end{theorem}
This probability is also the limit of other natural empirical measures on $\T$:
\begin{proposition} \label{other}
Let $\mu^{\leadsto}$ be the Lebesgue measure on $\R^+$ and $\mu^{\bullet}=\sum_{i=1}^{\infty} \delta_{Y_i}$. For every $l\in \R^+$ let $\mu^{\leadsto}_l$ (resp. $\mu^{\bullet}_l$) be the restriction of $\mu^{\leadsto}$ $(resp. \mu^{\bullet})$ to $\T_l=([0,l],d)$. Also let for every $l\in \R^+$, $p_l^{\leadsto}=\frac{\mu_l^{\leadsto}}{\mu^{\leadsto}_l[0,l]}$ and $p_l^{\bullet}=\frac{\mu_l^{\bullet}}{\mu^{\leadsto}_l[0,l]}$. Then 
\[ p^{\bullet}_l\limit^{weakly}_{l\to \infty} p \quad \text{ and }\quad p^{\leadsto}_l\limit^{weakly}_{l\to \infty} p. \]
\end{proposition}
Intuitively speaking this comes from the fact that $\mu$ "dictates" how segments are glued together so the convergence of $p_l$ implies the convergence of many others quantities.
\begin{remark} Proposition \ref{other} shows that $p$ corresponds to the probability measure introduced in Aldous Pitman \cite{IntroICRT1}. In particular independent leafs sampled by $p$ "behaves" like $(Y_i)_{i\in \N}$. (\cite{IntroICRT1} Corollary 8)  \end{remark}
Then we prove the conjecture of Aldous, Miermont, Pitman in \cite{ExcICRT} about compactness.
\begin{theorem} \label{THM2}
The ICRT is almost surely compact if and only if
\begin{equation} \int^{\infty} \frac{1}{l\E[\mu[0,l]]}<\infty. \label{1404} \end{equation}
\end{theorem}
\begin{remark}
The conjecture in \cite{ExcICRT} is based on a comparison between the ICRT and Levy trees introduced by Le Gall Le Jan \cite{IntroLevy1}. Levy trees are characterized by their Laplace exponent $\psi$ and are compact if and only if $\int^{\infty} \frac{1}{\psi(l)} <\infty$ (see \cite{Duquesne2}). The formulation of the conjecture in  \cite{ExcICRT} is based on an analog of the Laplace exponent in the setting of ICRT, which behaves like $l\E[\mu[0,l]]$ (see Lemma \ref{condition}) which turns out to be equivalent to \eqref{1404}.
\end{remark}
For the proof of Theorem \ref{THM2}, we first translate the condition in \eqref{1404} into a more convenient one: it turns out (Lemma \ref{condition}) that
\[ \int^{\infty} \frac{1}{l\E[\mu[0,l]]}<\infty \quad \text{if and only if} \quad \sum_{n=1}^{\infty} \frac{\log \X_{2^n} }{\X_{2^n}} <\infty,\]  where for every $l\in \R^+$, $\X_l$ is the real number such that $\E[\mu[0,\X_l]]=l$ (see Lemma \ref{2} for existence and uniqueness).

To prove that the condition is sufficient, we will upper bound the law of the distance between a random point in $\T_{\X_{2^n}}$ and its projection on $\T_{\X_{2^{n-1}}}$. We then use this bound to prove that  \[ d_H(\T_{\X_{2^n}},\T_{\X_{2^{n-1}}})\leq C\frac{\log \X_{2^n} }{\X_{2^n}}, \]
where $d_H$ denotes the Hausdorff distance on subsets of $\T$. For the Hausdorff topology, Cauchy sequences of compact sets converge toward a compact set so this proves that $\sum_{n=1}^{\infty} \frac{\log \X_{2^n} }{\X_{2^n}} <\infty$ implies that $\T$ is compact.

The fact that the condition is necessary follows from an adaptation of an argument of Amini, Devroye, Griffiths, Olver in \cite{Amini}. We show that, for some fixed constants $c,C\in (0,\infty)$ and for all $k$ large enough:
\[ c \sum_{n=k+1}^{\infty} \frac{\log \X_{2^n} }{\X_{2^n}} \leq d_H \left (\T,\T_{\X_{2^k}} \right ) \leq C \sum_{n=k+1}^{\infty} \frac{\log \X_{2^n} } {\X_{2^n}}. \]

We then proceed to the computation of some fractal dimensions.
\begin{theorem} \label{THM3}
Almost surely 
\[\dim_P(\T) =\overline{\dim}(\T) =1+\limsup_{l\to \infty} \frac{\log l}{\log \E[\mu[0,l]]}.\]
Furthermore if $\log l=\E[\mu[0,l]]^{o(1)}$ then 
\[\dim_H(\T)=\underline{\dim}(\T)=1+\liminf_{l\to \infty} \frac{\log l}{\log \E[\mu[0,l]]}.\]
\end{theorem}
\begin{remark} If one replaces $l\E[\mu[0,l]]$ by the Laplace exponent $\psi$ then one recovers the formulas for the fractal dimensions of Levy trees obtained by Duquesne and Le Gall \cite{Duquesne1}.
\end{remark}
To prove Theorem \ref{THM3}, it suffices by Lemma \ref{FalconPunch} to upper bound the Minkowski dimensions and to lower bound the Packing and Hausdorff dimension. To upper bound $\overline{\dim}(\T)$ and $\underline{\dim}(\T)$  we use some cover of $\T$ which relies on $\log l=\E[\mu[0,l]]^{o(1)}$. Then we derive the lower bound on $\dim_P(\T)$ and $\dim_H(\T)$ from Lemma \ref{Hausdorff}.
\section{Preliminaries}  \label{Section 4}
This section should be seen as a tool box: we gather here a collection of lemmas that will be used repeatidly throughout the paper. Most of them are straightforward.
\subsection{Fundamental properties of $\mu$}
\begin{lemma} \label{2} 
The map $l\to\E[\mu[0,l]]$ is differentiable and its derivative decreases to $\theta_0^2$ as $l\to \infty$ we thus have as $l\to \infty$:
\[ \E[\mu[0,l]]=\theta_0^2l+o(l). \]
\end{lemma}
\begin{proof} By Fubini's theorem,
\begin{equation} \E\left [\mu[0,l]-\theta_0^2l \right ]=\E \left [\sum_{i=1}^{\infty}\theta_i\1_{X_i\leq l} \right ]=\sum_{i=1}^{\infty}\theta_i\proba \left (X_i\leq l \right )=\sum_{i=1}^{\infty}\theta_i(1-e^{-\theta_il}).\label{107} \end{equation}
Each term of the sum is positive and increasing so we can differentiate term by term: 
\[ \frac{d}{dl}\E\left [\mu[0,l]-\theta_0^2l \right ]=\sum_{i=1}^\infty \theta_i^2e^{-\theta_il}.\]
Since $\sum_{i=1}^\infty \theta_i^2<\infty$, by bounded convergence the last term decreases to $0$ as $l\to \infty$.
\end{proof}
Lemma \ref{2} implies that the map $l\mapsto \E[\mu[0,l]]$ is strictly increasing, continuous, and diverges, so is invertible. Thus for every $l\in \R^+$, there is a well-defined real number $\X_l$ with $\E[\mu[0,\X_l]]=l$.
\begin{lemma} \label{5} 
We have almost surely
\[\mu[0,l]\asym_{l\to\infty} \E \left [\mu[0,l] \right ].\]
\end{lemma}
\begin{proof}  For every $l\in \R^+$ the variance of $\mu_l$ is given by: 
\[\Varr [\mu[0,l]]= \Varr \left [\theta_0^2l+\sum_{i=1}^{\infty}\theta_i\1_{X_i\leq l} \right ]=\sum_{i=1}^{\infty} \Varr \left [\theta_i\1_{X_i\leq l} \right ] \leq \sum_{i=1}^{\infty}\theta_i^2 \leq 1.\]
Therefore for every $n\in \N$,
\[ \proba  \left ( \left | \mu[0,\X_{n^2}] -\E \left  [\mu[0,\X_{n^2}] \right ] \right | > n \right ) \leq \frac{1}{n^2}. \]
By definition of $\X_n$ we deduce by the Borel--Cantelli lemma that for every $n$ large enough 
\[ n^2-n\leq \mu[0,\X_{n^2}] \leq n^2+n. \]
We thus have almost surely $\mu[0,\X_{n^2}] \sim \E[\mu[0,\X_{n^2}]]=n^2$. This result is then extended to every $l\in\R^+$ by monotonicity of  $l\mapsto \mu[0,l]$.
\end{proof}
Note that Lemmas \ref{2} and \ref{5} implies that for every $l$ large enough $\mu[0,l]\leq l$.

The following lemma should be seen as an estimate for the "density" and "jump" of $l\mapsto \mu[0,l]$. 
\begin{lemma} \label{6} 
Almost surely there exists $L_0\in \R$ such that for every $l\geq L_0$ and $0\leq \delta \leq l$,
\[ \mu[l,l+\delta] \leq 2 \delta \frac{\E \left [  \mu \left [0,l\right ] \right ]}{l}+\frac{13\log (l)}{l}. \]
\end{lemma}
\begin{proof} First let us prove a concentration inequality for $\mu[l,l+\delta]$. We have by Fubini's Theorem,
\begin{equation} \E \left [e^{\frac{l}{2}\mu[l,l+\delta] }\right ]= \E \left [e^{\frac{l}{2}\theta_0^2\delta}\prod_{i=1}^{\infty}e^{\frac{l}{2} \theta_i\1_{l\leq X_i\leq l+\delta}} \right ]=e^{\frac{l}{2}\theta_0^2\delta} \prod_{i=1}^{\infty} \left (1+  (e^{\frac{l}{2}\theta_i}-1)\proba \left (l\leq X_i\leq l +\delta \right )\right ). \label{WAFWAF}  \end{equation}
Furthermore we have for every $i\in \N$, since $X_i$ is an exponential random variable of parameter $\theta_i$,
\[ (e^{\frac{l}{2}\theta_i}-1)\proba \left (l\leq X_i\leq l +\delta \right )=(1-e^{-\frac{l}{2}\theta_i})\proba \left ( \frac{l}{2} \leq X_i\leq \frac{l}{2} +\delta \right )\leq \frac{l}{2}\theta_i\proba \left ( \frac{l}{2} \leq X_i\leq \frac{l}{2} +\delta \right ), \]
Therefore by \eqref{WAFWAF} and \eqref{107},
\begin{align} \E \left [e^{\frac{l}{2} \mu[l,l+\delta]}\right ] 
 & \leq  \exp \left (\frac{l}{2}\theta_0^2\delta+ \sum_{i=1}^{\infty} \frac{l}{2}  \theta_i \proba \left ( \frac{l}{2} \leq X_i\leq \frac{l}{2} +\delta\right ) \right ) \notag
\\ & =  \exp \left ( \frac{l}{2} \E \left [  \mu \left [\frac{l}{2},\frac{l}{2}+\delta\right ] \right ] \right ).\label{15101}
\end{align}
Moreover by Lemma \ref{2}, $t\mapsto \E[\mu[0,t]]$ is concave and increasing, hence,
\begin{equation} \E \left [  \mu \left [\frac{l}{2},\frac{l}{2}+\delta\right ] \right ] \leq \frac{2\delta}{l}\E \left [  \mu \left [0,\frac{l}{2}\right ] \right ] \leq  \frac{2\delta}{l} \E \left [  \mu \left [0,l\right ] \right ].\label{15102} \end{equation}
Finally it follows from Markov's inequality, \eqref{15101}, and \eqref{15102}  that for every $l,l',t\in \R^+$,
\begin{equation} \proba \left ( \mu[l,l+\delta ] \geq\frac{2\delta}{l}\E \left [  \mu \left [0,l\right ] \right ] +\frac{2t}{l}\right )\leq e^{-t}. \label{1310} \end{equation}

We now derive the desired result from \eqref{1310}. First by the Borel--Cantelli Lemma, there exists almost surely an $N\in \N$, such that for every $n\geq N$ and $n\leq m \leq 8n$,
\[  \mu[\sqrt{n},\sqrt{m}] \leq 2\frac{(\sqrt{m}-\sqrt{n})}{\sqrt{n}}\E \left [  \mu \left [0,\sqrt{n}\right ] \right ]+\frac{6\log (n)}{\sqrt{n}}. \]

Now fix $l\geq N+10$, $0\leq \delta \leq l$ then let $n:=\max\{i, \sqrt{i}\leq l\}$ and let $m:= \min\{i,  l+\delta\leq \sqrt{i}\}$. Since $l\mapsto \mu[0,l]$ is non decreasing, we have,
\begin{align} \mu[l,l+\delta]\leq \mu[\sqrt{n},\sqrt{m}] & \leq 2\frac{\sqrt{m}-\sqrt{n}}{\sqrt{n}}\E \left [  \mu \left [0,\sqrt{n}\right ] \right ]+12\frac{\log (\sqrt{n})}{\sqrt{n}} \notag
\\  & \leq 2 \frac{\delta+2/l}{l-1/l} \E \left [  \mu \left [0,l\right ] \right ]+12\frac{\log (l)}{l} . \label{1510manger}\end{align}
Finally by Lemma \ref{2}, $\E \left [  \mu \left [0,l\right ] \right ]=O(l)$ as $l\to \infty$ and the desired result follows from \eqref{1510manger}. 
\end{proof}
\subsection{Key results on cuts and sticks}
\begin{lemma} \label{8} 
Almost surely there exists $L_0\in \R^+$ such that for every $l\geq L_0$ there are at most $2l\mu[0,l]\leq 2l^2$ cuts on $[0,l]$.
\end{lemma}
\begin{proof} 
Conditionally on $\mu$, $\{Y_i\}_{i\in \N}$ is a Poisson point process with rate $\mu[0,l]dl$ so the number of cuts in $[0,l]$ is stochastically dominated by a Poisson random variable $\alpha$ with mean $l\mu[0,l]$ and for $l$ large enough
\[ \proba \left (\alpha\geq \frac 3 2 l\mu[0,l] \right)\leq \frac{1}{l^2}. \]
Thus by the Borel--Cantelli lemma almost surely for every $l\in \N$ large enough, there are at most $\frac{3}{2}l\mu[0,l]$ cuts on $[0,l]$. This can be easily extended to all $l\in \R^+$ large enough using Lemmas \ref{2} and \ref{5}. We omit the straightforward details.
\end{proof}
\begin{lemma} \label{9} 
Almost surely there exists $i_0\in \N$ such that for every $i\geq i_0$:
\[ l_{i+1} \leq \frac{5\log(Y_i)}{M_i}. \]
\end{lemma}
\begin{proof}
Because the cuts are made at rate $\mu[0,l]dl$, for every $i\in \N$, $(Y_{i+1}-Y_i)\mu[0,Y_i]$ is stochastically dominated an exponential random variable with mean one. Therefore
\[ \proba \left ((Y_{i+1}-Y_{i})\mu[0,Y_i]  \geq 2\log(i)\right )\leq 1/i^2. \]
So by the Borel--Cantelli lemma and Lemma \ref{8}, for every $i$ large enough,
\[ Y_{i+1}-Y_i \leq \frac{2\log(i)}{\mu[0,Y_i]} \leq \frac{2\log(2Y_i^2)}{\mu[0,Y_i]} \leq \frac{5\log(Y_i)}{\mu[0,Y_i]}. \qedhere  \]  
\end{proof}
\begin{lemma} \label{10} 
Almost surely there exists $L_0\in \R^+$ such that for all $l\geq L_0$ and $i\in \N$ with $Y_i\geq l$,
\[ m_{i+1}\leq \frac{\log^2{l}}{l}.\]
\end{lemma}
\begin{proof}
We have by Lemmas \ref{9}, \ref{5}, and \ref{6}, as $i\to \infty$,
\[ l_{i+1}\leq \mu\left [Y_i, Y_i+\frac{5 \log Y_i}{\mu[0,Y_i]} \right ] \leq O\left ( \frac{ \log Y_i}{\mu[0,Y_i]} \frac{\mu[0,Y_i]}{Y_i} +\frac{\log Y_i}{Y_i} \right )=o\left ( \frac{\log^2 Y_i}{Y_i}\right ).\qedhere \]
\end{proof}
\subsection{An estimate of distances in $\T$}
\begin{definition*} For every random variables $A$, $B$ on $\R$ we recall that $A$ is stochastically dominated by $B$ if and only if for every $t\in \R^+$, $\proba(A\geq t)\leq \proba(B\geq t)$. In this case we write $A\leq_{\st} B$. Also for every $l\in \R^+$, let $\Exp(l)$ denotes an exponential random variable of mean $l$. 
\end{definition*}
\begin{lemma} \label{pizza} For every $x,y \in \R^+$, conditionally on $\mu$,  $d(\T_x,y)\leq_{\st} \Exp  (\frac{4}{\mu[0,x]} )$.
\end{lemma}
\begin{remark}
Proving an equivalent of Lemma \ref{pizza} is crucial for each studies on stick-breaking constructions, notably for compactness \cite{Curien,Seni}  and convergence \cite{Aldous1,Uniform}. Although the proof presented below use strong property on $\mu$, more general methods can be found in \cite{Aldous1,Uniform,Curien,Seni}.  Finally, we believe that such methods can be useful for the study of several other classes of algorithms.
\end{remark}
\begin{proof} To simplify the notation let for every $l\in \R^+$, $\mathcal{F}_l:=\sigma \big (\mu, \big \{(Y_i,Z_i)\big \}_{i\in \N}\cap [l,+\infty]\times \R^+ \big )$. We first prove that if $\mu[0,x]\geq 2\mu[0,y)$ then conditionally on $\mathcal{F}_y$, $d(\T_x,y)\leq_{\st}\Exp (\frac{2}{\mu[0,x]})$. If $y\leq x$ then $d(\T_x,y)=0$. We assume henceforth that it is not the case. Let us "follow" the geodesic path from $y$ to $\T_x$. More precisely we define the following sequence by induction (see Figure \ref{ImageProof}). Let $z_0:= y$, then for every $i\geq 0$, let $k_i:=\max\{k\in \N: Y_{k}<z_i\}$ and let $y_i:=Y_{k_i}$, and $z_{i+1}:=Z_{k_i}$. Additionaly let $T$ denotes the smallest integer such that $z_{T+1}\leq x$. Note that 
\begin{equation} d(y,\T_x)=\sum_{i=0}^{T} \big (z_i-\max(y_i,x)\big). \label{07100} \end{equation}
 \begin{figure}[!h]  \label{ImageProof}
\centering
\includegraphics[scale=0.6]{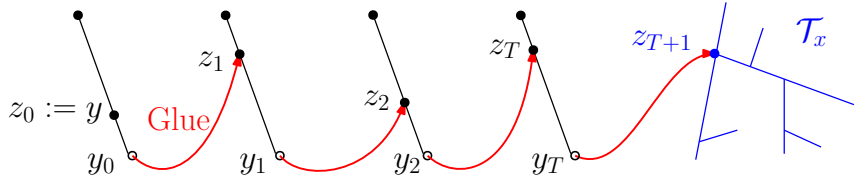}
\caption{A typical construction of $(y_i,z_i)_{i\in \N}$. Note that in general we do not know if $y_T\in \T_x$.} \label{ImageProof}
\end{figure}

  Now recall that conditionally on $\mu$, $\{(Y_i,Z_i),i\in \N\}$ is a Poisson point process, so $\{(y_i,z_i),\mathcal{F}_{y_i}\}_{i\geq 0}$ is a Markov chain. Also note that $T+1=\inf\{n: z_n<x\}$ is a stopping time for $\{(y_i,z_i),\mathcal{F}_{y_i}\}_{i\geq 0}$. 
Moreover, for every $i\in \N$ conditionally on $(\mu,y_i,z_i)$, $z_{i+1}$ has law $p_{y_i}$. Hence if $y_i\geq x$,
 \[ \proba (z_{i+1}\leq x | \mu,y_i,z_i ) =p_{y_i}[0,x] \geq \frac{\mu[0,x]}{\mu[0,y)}\geq \frac{1}{2}. \]
So $T$ is stochastically dominated by a geometric random variable of parameter $1/2$. Furthermore, if $i\leq T$, conditionally on $(\mu,y_i)$, $\{Y_i\}_{i\in \N}\cap[x,y_i)$ is a Poisson point process of rate $\mu[0,l]\geq \mu[0,x]$ so $z_i-\max(y_i,x)\leq_{\st}\Exp(\frac{1}{\mu[0,x]})$. Finally it follows from \eqref{07100} that $d(\T_x,y)\leq_{\st} \Exp(\frac{2}{\mu[0,x]}  ).$

Let us now treat the general case. As previously, we bound $d(\T_x,y)$ by following the geodesic path between $\T_x$ and $y$. More precisely, let for every $i\geq 0$, $x_i:=\inf\{a\in \R^+,\mu[0,a]\geq 2^i\mu[0,x]\}$ and let $y_i$ be the nearest point from $y$ on $[0,x_i]$. Note that 
\begin{equation} d(x,y)=\sum_{i=0}^{+\infty} d(y_{i},y_{i+1}). \label{24102} \end{equation}
Then for every $i\geq 0$, since $2\mu[0,x_i]\geq \mu[0,x_{i+1})$, the first case yields, conditionally on $\mathcal{F}_{y_{i+1}}$, 
\begin{equation} d(y_i,y_{i+1})=d(\T_{x_i},y_{i+1})\leq_{\st} \Exp \left (\frac{2}{\mu[0,x_i]} \right )\leq_{\st} \Exp \left (\frac{2^{1-i}}{\mu[0,x]} \right ). \label{24103} \end{equation}
Finally since for every $j>i$, $d(y_j,y_{j+1})$ is $\mathcal{F}_{y_{i+1}}$ measurable, it follows from \eqref{24102} and \eqref{24103} that 
$ d(x,y)\leq_{\st} \Exp (\frac{4}{\mu[0,x]})$.
 \end{proof}%
\section{The mass measure}  \label{Section 5}
 First we prove Lemma \ref{E=MC2} that describes precisely the evolution of the mass $\mu$ as we add branches to the tree. Then we prove that $(p_l)_{l\geq 0}$ is tight and use Lemma \ref{E=MC2} to prove that for every bounded Lipschitz function $(p_l(f))_{l\geq 0}$ converges.  It proves, by the Portmanteau Theorem, that $(p_l)_{l\geq 0}$ converges weakly toward a probability measure $p$ (Theorem \ref{THM1}). Then we adapt the argument to prove Proposition \ref{other}.
%
%
\subsection{The mass conservasion lemma} \label{5..2}
\begin{definition*} For every $l\in \R^+$ let the projection of $x$ in $\T_l$ be the nearest point from $x$ in $\T_l$. Also for every $S\subset\T$, let $S^{\uparrow l}$ be the set of $x\in\T$ such that the projection of $x$ in $\T_l$ is in $S$.
\end{definition*}
\begin{lemma}  \label{E=MC2} 
Almost surely $(\mu,(Y_i)_{i\in\N})$  satisfy the following property. For every $a$ large enough, conditionally on $\T_{Y_a}$, for every measurable set $S\subset \T_{Y_a}$, the following assertions hold.
\begin{compactitem}
\item[(i)] Almost surely $\{ p_l(S^{\uparrow Y_a} ) \}_{l\in \R^+}$ converges toward a real number $p(S^{\uparrow Y_a} )$.
\item[(ii)] If $\mu(S)\geq \frac{\log^6 Y_a}{Y_a}$ with probability at least $1-\frac{1}{Y_a^5}$, for every $l\geq Y_a$
\[  \left (1-\frac 1 {\log Y_a} \right )p_{Y_a} (S) \leq p_{l} \left (S^{\uparrow  Y_a} \right ) \leq  \left (1+\frac 1 {\log Y_a} \right) p_{Y_a} (S).\]
\item[(iii)] If $\mu(S)\leq \frac{\log^6 Y_a}{Y_a}$ with probability at least $1-\frac{1}{Y_a^5}$, for every $l\geq Y_a$
\[ p_{l} \left (S^{\uparrow  Y_a} \right ) \leq  \frac{\left (\log Y_a\right )^{6}}{Y_aM_a}.\] 
\end{compactitem} 
\end{lemma}
\begin{proof} First for every $i\geq a$, let $A_i:= \mu_{Y_i} \left (S^{\uparrow  Y_a} \right )$ and $\F_i:= \sigma \left (\mu, \{Y_n\}_{n\in \N}, \{Z_n\}_{1 \leq n < i} \right )$ . Note that for every $i\geq a$, since $Z_{i}$ has law $\frac{\mu_{Y_i}}{M_i}$, we have $(Y_i,Y_{i+1}]\subset S^{\uparrow  Y_a}$ with probability $\frac{A_i}{M_i}$ so
\[ \proba \left (A_{i+1}=A_i+m_{i+1} \right)=\frac{A_i}{M_i} \quad ; \quad \proba \left (A_{i+1}=A_i \right )=\frac{M_i-A_i}{M_i} .\]
Thus $(A_i,\F_i)_{i\geq a}$ can be seen as a P\'olya urn in the sense of Lemma \ref{P\'olya}. Furthermore by Lemma \ref{10}, we have almost surely for every $a$ large enough, $\max_{n>a} m_n\leq  \frac{\log^2 Y_a}{Y_a}$, hence by Lemma \ref{P\'olya} (b), for every $t\in [0,1]$,
\begin{align}  \proba \left ( \left . \sup_{i\geq a} \left | \frac{A_i}{M_i} - \frac{A_a}{M_a} \right  | >   t\frac{A_a}{M_a}  \right | A_a \right )
 \leq  2 \exp \left (-\frac{t^2}{8}\frac{A_a Y_a}{\log^2 Y_a} \right  )  .
\label{2610} \end{align}

Also still by Lemma \ref{10} we have for every $a\in \N$ large enough, $i\geq a$, and $Y_i\leq l \leq Y_{i+1}$,
\[p_l(S^{\uparrow Y_a})=\frac{\mu_l(S^{\uparrow Y_a})}{\mu[0,l]}\leq \frac{\mu_{Y_i}(S^{\uparrow Y_a})+m_{i+1}}{\mu[0,Y_i]} = \frac{A_i}{M_i}+\frac{m_{i+1}}{M_i}\leq \frac{A_i}{M_i}+\frac{\log^2 Y_a}{Y_aM_a}, \]
and similarly 
\[p_l(S^{\uparrow Y_a})=\frac{\mu_l(S^{\uparrow Y_a})}{\mu[0,l]}\geq \frac{\mu_{Y_{i+1}}(S^{\uparrow Y_a})-m_{i+1}}{\mu[0,Y_{i+1}]} = \frac{A_{i+1}}{M_{i+1}}-\frac{m_{i+1}}{M_{i+1}}\geq \frac{A_{i+1}}{M_{i+1}}-\frac{\log^2 Y_a}{Y_aM_a}. \]
Therefore,
\begin{equation} \sup_{l\geq Y_a} \left |  p_l(S^{\uparrow Y_a}) - \frac{A_a}{M_a} \right  | \leq \sup_{i\geq a} \left | \frac{A_i}{M_i} - \frac{A_a}{M_a} \right  | + \frac{\log^2 Y_a}{Y_aM_a}. \label{0610} \end{equation}
The claims in (i) (ii) (iii) are applications of the inequalities in \eqref{2610} and \eqref{0610}. 

Consider first (i). Note that \eqref{2610} implies that $\{\frac{A_i}{M_i}\}_{i\in \N}$ is almost surely Cauchy, and hence converges. Furthermore $\{\frac{\log^2 Y_a}{Y_aM_a}\}_{a\in\N}$ almost surely converges to 0. (i) then follows from \eqref{0610}.

Towards (ii), we have by assumption $A_a\geq \frac{\log^6 Y_a}{Y_a}$ so if $a\geq 10$, $\frac{\log^2 (Y_a)}{Y_aM_a}\leq \frac{1}{2\log Y_a}\frac{A_a}{M_a}$.
Therefore by \eqref{0610} it suffices to estimate the right-hand side of \eqref{2610} with $t=\frac{1}{2\log Y_a}$: 
\[2 \exp \left (-\frac{t^2}{8}\frac{A_a Y_a}{\log^2 Y_a} \right  )  \leq 2 \exp \left (-\frac{1}{32\log^2 Y_a}\frac{\log^{6} Y_a}{\log^2 Y_a} \right  )=o\left (\frac{1}{Y_a^5} \right )  \]
and $(ii)$ follows. $(iii)$ can be treated similarly using \eqref{2610} with $t= \frac{\log^6 Y_a}{2Y_aA_a}$. We leave the details to the reader. This concludes the proof.
\end{proof}
\subsection{Weak convergence of $\mu_n$ : proof of Theorem \ref{THM1}}
In this section we prove Theorem \ref{THM1}. Let us start with the tightness of $(p_l)_{l\in \R^+}$ which follows from the following lemma.
\begin{lemma} \label{tight} For every $n\in \N$  let $A_n$ be the set of $x\in \T$ such that, $d(x,[0,\X_{2^n}]) \leq 8n/2^{n}$ and $B_n:= \bigcap_{m \geq n} A_m$.
The following assertions hold: 
\begin{enumerate}
\item[(i)] Almost surely for every n large enough, for every $l\geq 0$, $p_l \left (B_n \right )\geq 1-2^{-2n}$.
\item[(ii)] For every $n$ large enough $B_n$ is compact.
\end{enumerate}
\end{lemma}
\begin{proof} First for every $n,m\in \N$, such that $m\geq n$,  conditionally on $\mu$ we have by Fubini's theorem, Lemma \ref{pizza}, and Lemma \ref{5},
\begin{equation*} \E \left [ \left . p_{\X_{2^m}}\left (\T\backslash A_n \right ) \right | \mu \right ] = \int_{0}^{\X_{2^m}} \proba\left( x\notin A_n \right ) \frac{d\mu(x)}{\mu[0,\X_{2^m}]} \leq e^{-\frac{8n}{2^n} \frac{\mu[0,\X_{2^{n}}]}{4}}=e^{-2n(1+o(1))}.\end{equation*}
It directly follows by Markov's inequality and the Borel--Cantelli lemma that almost surely for every $n$ large enough and $m\geq n$, \[p_{\X_{2^m}}\left (\T \backslash A_n \right ) \leq 2^{-2n-3}.\]
Therefore for every $n\in \N$ and $l\geq \X_{2^n}$, writing $k$ for the smallest integer such that $l\leq \X_{2^{k}}$ we have by Lemma \ref{5},
\[ p_{l}\left (\T \backslash A_n \right ) \leq \frac{\mu[0,\X_{2^k}] }{\mu[0,l]} p_{\X_{2^{k}}}\left (\T \backslash A_n \right ) \leq \frac{\mu[0,\X_{2^k}] }{\mu[0,\X_{2^{k-1}}]} p_{\X_{2^{k}}}\left (\T \backslash A_n \right ) \leq 2^{-2n-1}. \]
Note that the latter is also true for $l\leq \X_{2^n}$ since in this case $\T_l \subset \T_{\X_{2^n}}\subset A_n$. (i) then follows from a union bound on $n$.

Toward $(ii)$, note that $A_m$ is a closed set for $m\geq n$, so $B_n$ is a closed set as well. Therefore it suffices to show that any sequence $(x_i)_{i\in \N}$ in $B_n$ has an accumulation point. Fix $(x_i)_{i\in \N}$ then for every $m\in \N$ let $x_{i}^m$ be the projection of $x_i$ on $[0,\X_{2^m}]$. Since for every $m\in \N$, $\T_{\X_{2^m}}$ is compact, by a diagonal extraction procedure there exists an increasing function $\phi: \N \mapsto \N$ such that for every $m\in \N$, $(x_{\phi(i)}^m)_{i\in \N}$ converges. Hence, for every $m\geq n$ there exists $N\in \N$ such that for every $a,b\geq N$, $d(x_{\phi(a)}^m,x_{\phi(b)}^m )\leq 1/m $ and so 
\begin{align*} d(x_{\phi(a)},x_{\phi(b)}) \leq d(x_{\phi(a)}, x_{\phi(a)}^m)+ d(x_{\phi(a)}^m, x_{\phi(b)}^m)+d(x_{\phi(b)}^m, x_{\phi(b)})\leq \frac{8m}{2^m}+\frac{1}{m}+\frac{8m}{2^m} .\end{align*}
Therefore $(x_{\phi(i)})_{i\in \N}$ is Cauchy and thus converges since $\T$ is complete by definition. Since $(x_i)_{i\in \N}$ is arbitrary,  $B_n$ is compact.
\end{proof}
\begin{definition*}Let $\F$ be the set of positive, 1-Lipschitz functions that are bounded by 1 on $\T$.
For every finite measure $\nu$ on $\T$ and measurable function $f:\T\to \R$ let $\nu(f):=\int_{\T} f(x) d\nu(x)$. \end{definition*}
\begin{lemma} \label{manteau} Almost surely, for every $f\in \F$, $p_l(f)\limit p(f)$ as $l\to \infty$.
\end{lemma}
\begin{proof} First for every $a\in \N$ let $\{I_i^a\}_{1\leq i \leq N_a}$ be a partition of $\T_{Y_a}=([0,Y_a],d)$ into intervals of diameter at most $1/a$. Then for every $a\in \N$ and $1\leq i \leq N_a$ let $J_i^a:= (I^a_i)^{\uparrow Y_a}$ and let $x^a_i\in I^a_i$. Note that for every $a\in \N$, $\{J_i^a\}_{1\leq i \leq N_a}$ is a partition of $\T$. So for every $l\geq Y_a$ and $f\in \F$,
\begin{equation} p_l(f) = \sum_{i=1}^{N_a} p_l \left (\1_{J_i^a}  f \right) =  \sum_{i=1}^{N_a} p_l\left ( J_i^a \right )f(x_i^a)+ \sum_{i=1}^{N_a} p_l \left (\1_{J_i^a} \left (f-f(x_i^a)\right ) \right). \label{11101} \end{equation}
By Lemma \ref{E=MC2} (i), almost surely for every $f\in \F$ the first sum  converges toward $\sum_{i=1}^{N_a} p\left ( J_i^a \right )f(x_i^a)$ as $l$ goes to infinity. Let us bound the second sum in order to prove that $(p_l(f))_{l\in \R^+}$ is Cauchy. For every $a\in \N$ let $k_a$ be the largest integer such that $\X_{2^{k_a}}\leq Y_a$. We have for every $f\in \F$:
\[ \sum_{i=1}^{N_a} p_l \left (\1_{J_i^a}  \left (f-f(x_i^a)\right ) \right) =   \sum_{i=1}^{N_a} p_l \left (\1_{J_i^a\cap B_{k_a}} \left (f-f(x_i^a)\right ) \right) +  p_l \left (\1_{\T\backslash B_{k_a}}  \left (f-f(x_i^a)\right ) \right) \]
and  \[  \left | \sum_{i=1}^{N_a} p_l \left (\1_{J_i^a}  \left (f-f(x_i^a)\right ) \right)  \right | \leq \sum_{i=1}^{N_a} p_l \left (\1_{J_i^a\cap B_{k_a}}  \left | f-f(x_i^a) \right | \right)  + p_l(\T\backslash B_{k_a}). \]
Furthermore for every $a\in \N$ and $1\leq i \leq N_a$, recall that by definition $I^a_i$ has diameter at most $\frac{1}{a}$ and that $d_H([0,\X_{2^{k_a}}],B_{k_a})\leq 8k_a2^{-k_a}$.  Therefore $J_i^a\cap B_{k_a}=(I^a_i)^{\uparrow Y_a} \cap B_{k_a}$ has diameter at most $\delta_a:= \frac{1}{a}+\frac{16k_a}{2^{k_a}}$. Hence for every $f\in \F$,
\begin{equation*} \left | \sum_{i=1}^{N_a} p_l \left (\1_{J_i^a}  \left (f-f(x_i^a)\right ) \right)  \right | 
  \leq   \sum_{i=1}^{N_a} p_l \left (J_i^a \cap B_{k_a} \right ) \delta_a + p_l(\T\backslash B_{k_a}) \leq  \delta_a+p_l(\T\backslash B_{k_a}).\notag
\end{equation*}
Moreover by Lemma \ref{tight} for every $a$ large enough $p_l(\T\backslash B_{k_a})\leq 2^{-2k_a}$. Finally for every $f\in \F$,
\begin{equation} \limsup_{a\to \infty} \limsup_{l\to \infty} \left | \sum_{i=1}^{N_a} p_l \left (\1_{J_i^a} \left (f-f(x_i^a)\right ) \right) \right | =0, \label{0710} \end{equation}
which implies together with \eqref{11101} that $(p_l(f))_{l\in \R^+}$ is Cauchy and thus converges.
\end{proof}
\begin{proof}[Proof of Theorem \ref{THM1}] 
First by lemma \ref{tight}, $(p_l)_{l\in \R^+}$ is tight. The convergence of $(p_l)_{l\in \R^+}$ then directly follows from Lemma \ref{manteau} and the Portmanteau theorem.

Towards proving that $p$ has full support, we first prove that $\mu$ has almost surely full support. Note that it suffices to prove that for every $a<b\in \R^+$, almost surely $\mu[a,b]>0$. If $\theta_0>0$ then $\mu[a,b]\geq (b-a)\theta_0^2>0$. So we assume henceforth that $\theta_0=0$. Note that in this case, $\sum_{i=1}^\infty \theta_i=\infty$. Moreover, recall that $\{X_i\}_{i\in \N}$ is a family of independent exponential random variables of parameter $\{\theta_i\}_{i\in\N}$ so that,
\[ \sum_{i=1}^\infty \proba(X_i\in[a,b])=\sum_{i=1}^\infty e^{-\theta_i a}\left (1-e^{-\theta_i(b-a)} \right ) =\infty. \]
Therefore by the Borel--Cantelli lemma, for every $a,b\in \R^+$ almost surely there exists an $i\in \N$ such that $X_i\in [a,b]$ and so $\mu[a,b]\geq \theta_i>0$. Thus, $\mu$ has almost surely full support.

Next we prove that $p$ also has full support. Fix $x\in \R^+$ and $\e>0$. Additionally for every $a\in \N$ let $k_a$ be the largest integer such that $\X_{2^{k_a}}\leq Y_a$. Note that for every $a\in \N$ large enough, by definition of $B_{k_a}$, $B(x,\e)^{\uparrow Y_a}\cap B_{k_a}$ has diameter at most $\e+16k_a2^{-k_a}\leq 2\e$. It follows that,

\begin{equation} p \left (B(x,2\e)\right )\geq p \left (B(x,\e)^{\uparrow Y_a}\cap B_{k_a} \right ) \geq p \left (B(x,\e)^{\uparrow Y_a} \right ) - p(\T\backslash B_{k_a}). \label{07102}\end{equation}
On the one hand, recall that almost surely $\mu(B(x,\e))>0$. Thus  by Lemma \ref{E=MC2} (ii), for every $a$ large enough, with probability at least $1-1/Y_a^5$,
\[  p \left (B(x,\e)^{\uparrow Y_a} \right ) \geq \frac{1}{2}p_{Y_a} (B(x,\e)) =\frac{\mu_{Y_a}(B(x,\e))}{2M_a}. \]
On the other hand, by Lemmas \ref{tight} (i), \ref{5}, and the definition of $k_a$, for every $a$ large enough, 
\[ p(\T\backslash B_{k_a})\leq 2^{-2k_a} \leq 2\mu[0,\X_{2^{k_a}}]^{-2}\leq 2\mu[0,Y_a]^{-2} =o\left (1/M_a\right ).\] 
Therefore by \eqref{07102}, almost surely $p(B(x,2\e))>0$. Since $x$, $\e$ were arbitrary and since rational numbers are dense on $\T$, it follows that $p$ has full support.

Finally, we prove that almost surely $p$ gives measure $1$ to the set of leaves and is non-atomic. 
For every $\e>0$ and $S\subset \T$, let $B(S,\e)=\{x\in \T: d(x,S)<\e\}$. Then let $(\e_a)_{a\in\N}$ be a sequence of positive real numbers decreasing sufficiently fast so that for every $a>0$ and $0\leq i <a$ we have $\mu_{Y_a}(B((Y_i,Y_{i+1}],\e_a))\leq 2 \mu_{Y_a}(Y_a,Y_{a+1}]$. By Lemma \ref{E=MC2} (ii) (iii), for every $a$ large enough and $0\leq i<a$, with probability at least $1-1/Y_a^{5}$, for every $l\geq Y_a$,
\begin{equation} p_{l} \left (B((Y_i,Y_{i+1}],\e_a)^{\uparrow Y_a} \right) \leq \max \left \{ 2 p_{Y_a}  \Big (B((Y_i,Y_{i+1}],\e_a) \Big);  \frac{\left (\log Y_a\right )^{6}}{Y_aM_a}\right \}. \label{0810} \end{equation}
Since by Lemma \ref{8} $a=O(Y_a^2)$, the Borel--Cantelli lemma implies that almost surely \eqref{0810} is true for every $a$ large enough, $0\leq i<a$ and $l\geq Y_a$. Furthermore by Lemma \ref{10} $M:=\max_{i\in \N} \mu(Y_i,Y_{i+1}]<\infty$. Also note that $\frac{\left (\log Y_a\right )^{6}}{Y_a}\limit 0$ as $a\to +\infty$. Therefore for every $a$ large enough, $0\leq i<a$, and $l\geq a$:
\begin{equation} p_{l} \left (B((Y_i,Y_{i+1}],\e_a)^{\uparrow Y_a} \right) \leq \frac{4M}{M_a}. \label{08102}\end{equation}
Moreover since for every $a\in \N$ the projection on $\T_{Y_a}$ (see \ref{5..2} for definition) is a continuous fonction, for every $0\leq i <a$, $B((Y_i,Y_{i+1}],\e_a)^{\uparrow Y_a}$ is open. Thus by letting $l\to \infty$ in \eqref{08102},  the Portmanteau theorem yields for $a$ large enough: 
\begin{equation} p\left (B((Y_i,Y_{i+1}],\e_a)^{\uparrow Y_a} \right) \leq \frac{4M}{M_a}, \label{2102} \end{equation}
which tends to $0$ as $a\to\infty$. So for every $i\in \N$, $p[Y_i,Y_{i+1}]=0$. Summing over all $i\in \N$ we get $p(\R^+)=0$ and so $p$ gives measure $1$ to the set of leaves. Note that \eqref{2102} also yield for every $a\in \N$,
\[ \sup_{x\in \T} p\{x\} = \max_{0\leq i <a}  \sup_{x\in [Y_i,Y_{i+1}]^{\uparrow Y_a}} p\{x\} \leq  \frac{4M}{M_a},\]
which implies, taking $a\to \infty$, that $p$ is non-atomic.
\end{proof}

\subsection{Other convergences toward $p$ : proof of Proposition \ref{other}}
In this section we prove Proposition \ref{other}. We will in fact prove the  following stronger result.

\begin{lemma} \label{other2}Let $\mu^{\alpha}$ be a positive random Borel measure on $\R^+$ which is $\sigma(\mu,\{Y_i\}_{i\in \N})$ measurable. Let for every $l\in \R^+$, $\mu^{\alpha}_l$ be the restriction of $\mu^{\alpha}$ to $\T_l=([0,l],d)$ and $p^{\alpha}_l:=\frac{\mu^{\alpha}_l}{\mu^{\alpha}[0,l]}$. Suppose that almost surely the following assertions hold: 
\begin{compactitem}
\item[(i)] For every $l>0$ $\mu^{\alpha}[0,l]<\infty$, and $\mu^{\alpha}(\R^+)=+\infty$.
\item[(ii)] There exists $\e>0$ such that $\mu^{\alpha}(Y_{i-1},Y_{i}]=o(\mu^{\alpha}[0,Y_i]^{1-\e})$. 
\item[(iii)] For all $\e>0$, $\sum_{i=1}^{n} \mu^{\alpha}(Y_{i-1},Y_i]\1_{Y_i-Y_{i-1}>\e} =o(\mu^{\alpha}[0,Y_n])$. \end{compactitem}
Then almost surely $\{p^{\alpha}_l\}_{l\in \R^+}$ converges weakly toward $p$.
\end{lemma}
In order to prove Lemma \ref{other2}, we first show the following strong law of large number.
\begin{lemma} \label{strong2} Let $\mu^{\alpha}$ be such as in Lemma \ref{other} and $S\subset \T$ be a random measurable set such that for every $n$ large enough, $p_{Y_n}(S)$ is $\sigma(\mu,\{Y_i\}_{i\in \N}, \{Z_i\}_{1\leq i <n})$ measurable. We have almost surely,
\[  \limsup_{n\to \infty} \sum_{i=1}^n p^{\alpha}(Y_{i-1},Y_{i}]  \1_{Z_{i-1}\in S} \leq \limsup_{l\to \infty} p_l(S).\]
\end{lemma}
\begin{proof}Let $\{U_n\}_{n\in \N}$ be a family of independent uniform random variables on $[0,1]$. Since for every $n\in \N$, conditionally on $(\mu,\{Y_j\}_{j\in \N}, \{Z_j\}_{1\leq j <n} )$, $Z_n$ has law $p_{Y_n}$, we may couple $\{U_n\}_{n\in \N}$ and $\{Z_n\}_{n\in \N}$ in such a way that for every $n$ large enough, $Z_n\in S$ if and only if $U_n\leq p_{Y_n}(S)$. Therefore, by Lemma \ref{strong} and assumptions $(i)$ and $(ii)$, almost surely for every $t>\limsup_{l\to \infty} p_l(S) $, 
\begin{equation}\limsup_{n\to \infty} \frac{\sum_{i=1}^n \mu^{\alpha}(Y_{i-1},Y_{i}]\1_{Z_{i-1}\in S}}{\mu^{\alpha}[0,Y_n]} \leq \limsup_{n\to \infty}\frac{\sum_{i=1}^n \mu^{\alpha}(Y_{i-1},Y_{i}]\1_{U_{i-1}\leq t}}{\mu^{\alpha}[0,Y_n]}=t.  \label{cle} \end{equation}
Taking $t\to \limsup_{l\to \infty} p_l(S)$ in \eqref{cle} yields the desired inequality. 
\end{proof}

\begin{proof}[Proof of Lemma \ref{other2}] First by the Portmanteau's theorem it suffices to prove that for every $f\in \F$, $p_l^\alpha(f)\to p(f)$ where $\F$ is the set of positive, 1-Lipschitz functions that are bounded by 1 on $\T$. Moreover since we work with probability measures and since for every $f\in \F$, $(1-f)\in \F$, it suffices to prove instead that for every $f\in \F$, $\limsup p_l^\alpha(f)\leq p(f)$. To this end, we proceed as in the proof of Lemma \ref{manteau} and will hence use the same notations. In addition, for $\e>0$ let $\Lambda_{\e}:=\bigcup_{i, Y_{i+1}-Y_i\leq \e} (Y_{i},Y_{i+1})$ and for every $x\in \T$, let $\zeta(x):=Z_{\max\{i, Y_i\leq x\}}$. 

Now fix $\e>0$ and recall from the proof of Lemma \ref{manteau} that for every $a\in \N$, $\{J_i^a\}_{1\leq i \leq N_a}$,  is a partition of $\T$, so for every $f\in\F$ and $l\geq Y_a$,
\begin{equation} p_l^{\alpha}(f) \leq \sum_{i=1}^{N_a} p_l^{\alpha}(f\1_{z(\cdot)\in J_i^a\cap B_{k_a}}\1_{\Gamma_e})+p_l^{\alpha}(f\1_{z(\cdot)\notin B_{k_a}}) + p_l^{\alpha}(f\1_{\T\backslash \Gamma_\e}).\label{0910} \end{equation}

We now upper bound each term of \eqref{0910} separately. First, recall that $J_i^a\cap B_{k_a}$ have diameter at most $\delta_a$, thus for every $1\leq i \leq N_a$ and $s\in S_i^a:=\{x,z(x)\in J_i^a\cap B_{k_a}\}\cap\Gamma_\e$ we have $d(s,x_i^a)\leq \e+\delta_a$. Therefore for every $f\in \F$,
\begin{equation}
 \sum_{i=1}^{N_a} p_l^{\alpha}(f\1_{S_i^a}) \leq \sum_{i=1}^{N_a} p_l^{\alpha}((f(x_i^a)+\e+\delta_a)\1_{S_i^a}) \leq \sum_{i=1}^{N_a} f(x_i^a)p_l^{\alpha}\left (\1_{z(\cdot)\in J_i^a}\right )+\e +\delta_a. \label{2010}
 \end{equation}
Furthermore by Lemma \ref{E=MC2} (i) almost surely for every $a\in \N$ and $1\leq i \leq N_a$, $p_l\left (J_i^a\right )\limit p(J_i^a)$ as $l\to \infty$, hence by Lemma \ref{strong2} almost surely 
\[ \limsup_{l\to \infty} p_l^{\alpha}(\1_{z(\cdot)\in J_i^a})\leq p(J_i^a).\]
Therefore since $\delta_a\to 0$ as $a\to \infty$, we have by \eqref{2010} and \eqref{0710}  for every $f\in \F$,
\[ \limsup_{a\to \infty} \limsup_{l\to \infty} \sum_{i=1}^{N_a} p_l^{\alpha}(f\1_{S_i^a})\leq \limsup_{a\to \infty} \sum_{i=1}^{N_a} f(x_i^a)p(J_i^a) +\e \leq  p(f)+\e.\]
Next we have by Lemma \ref{tight} (i) and Lemma \ref{strong2}, almost surely for every $a$ large enough, 
\[ \limsup_{l\to \infty} p_l^{\alpha}(\1_{z(\cdot)\notin B_{k_a}}) \leq 2^{-2k_a}. \] Futhermore by assumption (iii), $p_l^{\alpha}(\1_{\T\backslash \Gamma_\e})\to 0$ as $l\to 0$. Finally \eqref{0910} yields, for every $f\in\F$,
\[  \limsup_{l\to \infty} p_l^{\alpha}(f)\leq  \limsup_{a\to \infty} p(f)+ \e+2^{-2k_a} = p(f)+\e. \]
Taking $\e\to 0$ in the previous inequality concludes the proof.
\end{proof}
\begin{proof}[Proof of Proposition \ref{other}]
We now justify that $\mu^{\leadsto}$ and $\mu^{\bullet}$ satisfy the assumptions of Lemma \ref{other2}. First (i) and the $\sigma(\mu,\{Y_i\}_{i\in \N})$ measurability for $\mu^{\leadsto}$ and $\mu^{\bullet}$ are straightforward from their definitions. (ii) for $\mu^{\leadsto}$ is an immediate consequence of Lemma \ref{9}. (ii) for $\mu^{\bullet}$ comes from $\mu^{\bullet}(Y_n,Y_{n+1}]=1$. (iii) is a little tedious to prove and follows directly from the fact that conditionally on $\mu$, $\{Y_i\}_{i\in \N}$ is a Poisson point process with rate $\mu[0,l]dl$ and that by Lemma \ref{5} almost surely $\mu[0,l]\limit \infty$ as $l\to \infty$. We omit the details. This concludes the proof of Proposition \ref{other}.
\end{proof}

\section{Compactness} \label{Compactness}  \label{Section 6}
\subsection{Equivalent condition}
In this section, we obtain a condition equivalent to that of Theorem \ref{THM2} which is more convenient to study the compactness of the ICRT from the bounds provided by Lemmas \ref{BIG LEMMA} and \ref{manger}. Additionally we also prove that the condition conjectured in \cite{ExcICRT} is also equivalent to that of Theorem \ref{THM2}. For $l\geq 0$, recall that $\X_l$ is defined by $\E[\mu[0,\X_l]]=l$ and let 
\[\psi(l):= \frac{\theta_0^2}{2} l^2+ \sum_{i=1}^{\infty} (e^{-l\theta_i}-1+l\theta_i). \] 
 \begin{lemma} \label{condition} The following conditions are equivalent:
\[ (i) \quad \int^{+\infty} \frac{dl}{l\E[\mu[0,l]]} <+\infty \quad , \quad (ii) \quad \int^{+\infty} \frac{dl}{\psi(l)} <+\infty \quad , \quad(iii) \quad \sum^{+\infty} \frac{\log \X_{2^n}}{2^n}<\infty.  \]
\end{lemma}
\begin{proof}
Since for every $x\in \R^+$, $e^{-x}-1+x\leq x(1-e^{-x}) \leq 2 \left ( e^{-x}-1+x \right )$, for every $l\geq 0$:
\[ \frac{\theta_0^2}{2} l^2+ \sum_{i=1}^{\infty} \left (e^{-l\theta_i}-1+l\theta_i \right ) \leq  \theta_0^2 l^2+\sum_{i=1}^{\infty}l\theta_i \left (1-e^{-\theta_il} \right )  \leq  \theta_0^2 l^2+ \sum_{i=1}^{\infty} 2 \left(e^{-l\theta_i}-1+l\theta_i \right). \]
So by  \eqref{107} for every $l\geq 0$, $\psi(l)\leq l\E[\mu[0,l]]\leq 2\psi(l)$. It follows readily that (i) and (ii) are equivalent. Furthermore
\[ \int_{\X_1}^{\infty} \frac{dl}{l\E[\mu[0,l]]} = \sum_{k=0}^{\infty} \int_{\X_{2^k}}^{\X_{2^{k+1}}} \frac{dl}{l\E[\mu[0,l]]} \leq \sum_{k=0}^{\infty} \int_{\X_{2^k}}^{\X_{2^{k+1}}} \frac{dl}{l2^k} 
= \sum_{k=1}^{\infty} \frac{\log \X_{2^k}}{2^k}-\log \X_1, \]
and similarly
\[ \int_{\X_1}^{\infty} \frac{dl}{l\E[\mu[0,l]]} = \sum_{k=0}^{\infty} \int_{\X_{2^k}}^{\X_{2^{k+1}}} \frac{dl}{l\E[\mu[0,l]]} \geq \sum_{k=0}^{\infty} \int_{\X_{2^k}}^{\X_{2^{k+1}}} \frac{dl}{l2^{k+1}}= \sum_{k=1}^{\infty} \frac{\log \X_{2^k}}{2^{k+1}}-\frac{\log \X_1}{2}. \]
 So (i) and (iii) are equivalent.
\end{proof}
\subsection{The condition of Theorem \ref{THM2} is sufficient for compactness}
The aim of this section is to prove Lemma \ref{BIG LEMMA} below. This Lemma implies that under condition (iii) of Lemma \ref{condition}, $(\T_{\X_{2^k}})_{k\in \N}$ is a Cauchy sequence of compact sets for the Hausdorff topology and thus converges toward a compact set. Since $(\T_{\X_{2^k}})_{k\in\N}$ is increasing (for $\subset$) toward $\T$, $\T$ is the only possible limit, and hence is compact.

\begin{lemma} \label{BIG LEMMA}
Almost surely, for every $k$ large enough: 
\[d_{H}(\T_{\X_{2^{k-1}}}, \T_{\X_{2^{k}}})\leq 21 \frac{\log \X_{2^k}}{2^k}. \]
\end{lemma}
\begin{proof} For every $k\in \N$ and $x\in \T$, let $E_k(x)$ denotes the event $d(x,[0,\X_{2^{k-1}}])>20 \log \X_{2^k} 2^{-k}$.
First by Fubini's theorem and Lemma \ref{pizza}, we have conditionally on $\mu$:
\[ \E \left [ \left . \int_{0}^{\X_{2^k}} \1_{E_k(x)} dx \right | \mu \right  ]  = \int_{0}^{\X_{2^k}} \proba \left ( \left . E_k(x) \right | \mu \right ) dx  \leq \X_{2^k} \exp \left ( -5 \frac{\log \X_{2^k}}{2^k} \mu[0,\X_{2^{k-1}}] \right ).
\]
Then by Lemma \ref{5} as $k$ goes to infinity $\mu[0,\X_{2^{k}}]\sim 2^k$. So for every $k$ large enough:
\[  \E \left [ \left . \int_{0}^{\X_{2^k}} \1_{E_k(x)} dx \right | \mu \right  ]  \leq \X_{2^k}^{-4/3}. \]
Furthermore by Lemma \ref{2}, $2^k=O(\X_{2^k})$ so $\sum \X_{2^k}^{-1/3}<\infty$. Hence by Markov's inequality and the Borel--Cantelli lemma, for every $k$ large enough:
\[  \int_{0}^{\X_{2^k}} \1_{E_k(x)}dx < \X_{2^k}^{-1}. \]
Note that it implies that, for every $k$ large enough and $x\in [0,\X_{2^k}]$,  
\[ d(x,[0,\X_{2^{k-1}}])\leq 20 \frac{\log \X_{2^k}}{2^k}  +\X_{2^k}^{-1}, \]
since otherwise the geodesic path from $x$ to $[0,\X_{2^{k-1}}]$ would contain a segment $S$ of length at least $\frac{1}{\X_{2^k}}$ such that $d(S,[0,\X_{2^{k-1}}])> 20\log \X_{2^k} 2^{-k}$. Finally by Lemma \ref{2}, for every $k$ large enough $\X_{2^k}\geq 2^k$, hence $\X_{2^k}^{-1} \leq \log \X_{2^k} 2^{-k}$. This concludes the proof.
\end{proof}
\subsection{The condition of Theorem \ref{THM2} is necessary for compactness}
The following section is organized as follow: Lemma \ref{Long} defines and proves the existence of "long" segments, Lemma \ref{glue} proves that they tend to "aggregate". Lemma \ref{manger} deduces a lower bound on $d_H(\T_{\X_{2^k}}, \T)$ from the two previous lemmas, thus proving that the condition is necessary. Finally Lemma \ref{NonCompact} gives a more precise view of the geometry of the tree in the non-compact case: "the tree is infinite in every direction".
\begin{lemma} \label{Long} For every $n\in \N$ let $L_n:=\frac{\log \X_{2^n}}{2^{n+2}}$ and let $\I_n$ be the set of segments $[Y_a+L_n, Y_{a+1}]$ with 
\[ Y_a\in [\X_{2^n},\X_{2^{n+1}}) \quad ; \quad Y_a+L_n\leq Y_{a+1} \quad ; \quad \mu \left [Y_a+L_n, Y_{a+1} \right ] \geq  \frac{1}{\X_{2^{n+1}}^2}. \]
Almost surely for every $n$ large enough we have $ \# \I_n\geq 2^{n+2}\X_{2^n}^{1/3}$.
\end{lemma}
\begin{proof}
Write $ \I'_n$ for the set of segments $[Y_a+L_n, Y_{a+1}]$ with
\[ Y_a\in [\X_{2^n},\X_{2^{n+1}}) \quad ; \quad Y_a+L_n\leq Y_{a+1} \quad ; \quad \mu \left [Y_a+L_n, Y_{a+1} \right ] <  \frac{1}{\X_{2^{n+1}}^2}. \]
First by Lemmas \ref{8} and \ref{5}, for every $n$ large enough, there are at most $2^{n+2}\X_{2^{n+1}}$ cuts on $\left  [0,\X_{2^{n+1}} \right ]$, hence $\#  \I'_n \leq 2^{n+2}\X_{2^{n+1}}$. Furthermore, by Lemma \ref{10},  for every $n$ large enough and $I\in\I_n$, we have $\mu( I) \leq \log^2 \X_{2^n}/ \X_{2^n}$ so
\[
\sum_{I\in \I_n\cup  \I'_n} \mu \left ( I \right )
 =\sum_{I\in \I_n} \mu \left ( I \right ) +\sum_{I\in  \I'_n} \mu \left ( I \right )
 \leq \# \I_n\frac{\log^2 \X_{2^n}}{\X_{2^n}}+ \frac{2^{n+2}}{\X_{2^{n+1}}} . 
\]
Therefore, since $\X_{2^{n+1}}\geq \X_{2^n}$, it suffices to prove that, writing $S_n:= \bigcup_{I\in \I_n\cup \I'_n} I$,
\begin{equation} \mu \left ( S_n\right ) > 2^{n+2}{\X_{2^n}^{-2/3}}{\log^2 \X_{2^{n}}}+\frac{2^{n+2}}{\X_{2^n}}\label{0210}. \end{equation} 

Note that for every $x\in [\X_{2^n},\X_{2^{n+1}}]$, $x\in S_n$ if and only if there is a cut in $[\X_{2^n},x]$ and no cut in $[x-L_n,x]$. So if there is a cut in $[\X_{2^n},\X_{2^n}+1]$,
 \[ \mu \left ( S_n \right )
 \geq
  \int_{\X_{2^n}+L_n+1}^{\X_{2^{n+1}}} \emptyset_{x-L_n,x}d\mu(x), \] 
  where for every $x\leq y$, $\emptyset_{x,y}:=\1_{\forall i\in \N,\,Y_i\notin [x,y]}$.
 Let $A_n$ denotes the right-hand side above. Since, conditionally on $\mu$, $(Y_i)_{i\in \N}$ is a Poisson point process of rate $\mu[0,l]dl$, for every $n$ large enough
 \[ \proba \left ( \left . A_n >\mu \left ( S_n \right ) \right | \mu\right )  \leq  \proba \left ( \left . \emptyset_{\X_{2^n},\X_{2^n}+1}=0 \right | \mu\right )
\leq e^{-\mu \left[0,\X_{2^n} \right ]}
 \leq  e^{-2^{n-1}}. \]
Therefore, by the Borel--Cantelli lemma, almost surely for every $n$ large enough $\mu \left ( S_n \right )\geq A_n$.
 
We now lower bound $A_n$ via a second moment method. We have, still by the properties of $(Y_i)_{i\in \N}$,
\begin{equation} \E[ A_n|\mu] \geq  \int_{\X_{2^n}+L_n+1}^{\X_{2^{n+1}}}  e^{- \mu \left[0,\X_{2^{n+1}} \right ]L_n}d\mu(x) =  e^{-\mu \left[0,\X_{2^{n+1}} \right ]L_n }\mu \left [\X_{2^n}+L_n+1,\X_{2^{n+1}}\right ].\label{18002}
 \end{equation}
 Furthermore note that $\frac{1+L_n}{\X_{2^n}}\to0$  as $n\to \infty$, hence by Lemmas \ref{5} and \ref{2} almost surely as $n\to \infty$,
 \begin{equation*}\mu[\X_{2^n}+L_n+1,\X_{2^{n+1}}] = \E[\mu[0,\X_{2^{n+1}}]](1+o(1))-\E[\mu[0,\X_{2^n}+L_n+1]](1+o(1)) 
 \sim  2^n.
  \end{equation*}
It follows from \eqref{18002}, Lemma \ref{5}, and the definition of $L_n$ that, as $n\to \infty$,
\begin{equation} \E[A_n| \mu] \geq \X_{2^n}^{-1/2+o(1)}  2^n. \label{1800} \end{equation}

Moreover we have by Fubini's theorem,
\[ \Varr[ A_n|\mu]=\int_{\X_{2^n}+L_n+1}^{\X_{2^{n+1}}}\int_{\X_{2^{n}}+L_n+1}^{\X_{2^{n+1}}} \Covv \left [ \left . \emptyset_{x-L_n,x}\emptyset_{y-L_n,y} \right | \mu \right ]d\mu(x)d\mu(y). \]
Note that for every $x,y\in \R^+$, $\Covv \left [ \left . \emptyset_{x-L_n,x} \emptyset_{y-L_n,y} \right | \mu \right ]\leq \E \left [ \emptyset_{y-L_n,y}  | \mu  \right ]$, and that conditionally on $\mu$, $\emptyset_{x-L_n,x}$ and $\emptyset_{y-L_n,y}$ are independent when $|y-x|>L_n$. It follows that,
\begin{align} \Varr[ A_n|\mu] & \leq  \int_{\X_{2^n}+L_n+1}^{\X_{2^{n+1}}}  \E \left [ \emptyset_{y-L_n,y}  | \mu  \right ] \int_{y-L_n}^{y+L_n} d\mu(x)d\mu(y) \notag
\\ & \leq \E[A_n|\mu] \max_{\X_{2^n}+L_n+1\leq y \leq \X_{2^{n+1}}} \mu[y-L_n,y+L_n]. \label{1801}\end{align}
Furthermore by Lemma \ref{6}, for every $n$ large enough and $y\in [\X_{2^n}+1,\X_{2^{n+1}}-L_n]$,
 \begin{equation}  \mu[y,y+2L_n] \leq 4L_n \frac{\E[\mu[0,y]]}{y} + 13\frac{\log y}{y}
  \leq 4 \frac{\log \X_{2^n}}{2^{n+2}} \frac{\E[\mu[0,\X_{2^{n+1}}]]}{\X_{2^n}}+13\frac{\log \X_{2^n}}{\X_{2^n}}. \label{1802} \end{equation}
Put together \eqref{1801} and \eqref{1802} yield as $n\to \infty$,
\begin{equation} \Varr[ A_n|\mu] \leq \E[A_n|\mu]  \frac{17\log \X_{2^n}}{\X_{2^n}}. \label{1610} \end{equation}
Therefore, by Chebyshev's inequality,  \eqref{1610}, and \eqref{1800}, we have as $n\to \infty$,
 \begin{equation*} \proba\left ( \left .A_n \leq \frac{\E[A_n|\mu]}{2}  \right | \mu \right ) 
  \leq 4\frac{\Varr[ A_n|\mu]}{\E[ A_n|\mu]^2} \leq \frac{O(1)\log \X_{2^n}}{\X_{2^n}\E[ A_n|\mu]}\leq  2^{-n}\X_{2^n}^{-1/2+o(1)}.
 \end{equation*} 
 So by the Borel--Cantelli lemma, almost surely for every $n$ large enough $A_n \geq \frac{\E[A_n|\mu]}{2}$. 
Finally the inequality in \eqref{0210} follows from \eqref{1800} and the fact that for every $n$ large enough $\mu(S_n)\geq A_n$. This concludes the proof.
 \end{proof}

Formally we call the segments in $\bigcup_{n\in \N} \mathcal{I}_n$ "long". The following lemma proves that those long segments tend to "glue" to one another.
\begin{lemma} \label{glue} For every $I\in \bigcup_{n\in \N} \I_n$ let $a_I$ denotes the only integer such that $I\subset (Y_{a_I},Y_{a_I+1}]$. Almost surely for every $n,m\in \N$ large enough with $n<m$ and $\X_{2^m}\geq \X_{2^{n+1}}^8$, for every $I\in \I_{n}$ there exists $I' \in \I_{m}$ such that $Z_{a_{I'}}\in I$. In this case we say that  $I'$ is glued on $I$. 
\end{lemma}
\begin{proof}
Conditionally on $\f:=\sigma(\mu,(Y_i)_{i\geq 1})$, $(Z_i)_{i\geq 1}$ are independent random variables with law $(p_{Y_i})_{i\in \N}$ so for every $i\in \N$ and $I\in \I_{n}$
\begin{align*} \proba \left ( \left. \forall I' \in \I_{m}, Z_{a_{I'}}\notin I \right | \f \right )
 =  \prod_{I'\in \I_{m}}  \left (1-\frac{\mu(I)}{\mu[0,Y_{a_{I'}}]} \right ) 
 \leq  \exp \left (-\#\I_{m} \frac{\mu(I)}{\mu\left [0,\X_{2^{m+1}} \right ]} \right ).
\end{align*}
Furthermore we have by definition of $I_{n}$, $\mu(I)\leq \X_{2^{n+1}}^{-2}\leq \X_{2^m}^{-1/4}$. It follows from Lemmas \ref{Long} and \ref{5} that for every $m$ large enough, 
\[ \proba \left ( \left. \forall I' \in \I_{m}, Z_{a_{I'}}\notin I \right | \f \right ) \leq \exp \left (-\X_{2^m}^{-1/4}\frac{ 2^{m+2}\X_{2^m}^{1/3} }{2^{m+2}} \right )=  \exp \left( -\X_{2^m}^{1/12} \right ). \]
Moreover, by Lemma \ref{8}, for every $i$ large enough $\#\I_{n}\leq 2\X_{2^{n+1}}^2$, and by Lemma \ref{2} for every $m\in \N$, $\X_{2^{m}}\geq 2^{m}$. So for every $m$ large enough,
\[\proba \left ( \left. \exists I\in \I_{n}, \, \forall I' \in \I_{m}, Z_{a_{I'}}\notin I \right | \f \right ) \leq 2\X_{2^{n}}^2 e^{-\X_{2^{m}}^{1/12}} \leq f(\X_{2^m})\leq f(2^m), \]
where $f:x\mapsto 2x^2e^{-x^{1/12}}$. Since $\sum_{n=0}^{\infty}\sum_{m=n}^{\infty}f(2^m)<\infty$  the Borel--Cantelli lemma yields the desired result.
\end{proof}

\begin{lemma} \label{manger}
Almost surely for every $k$ large enough:
\[d_{H}(\T_{\X_{2^k}}, \T)\geq \frac{1}{128} \sum_{n=k}^{\infty} \frac{\log \X_{2^{n}}}{2^n}. \]
\end{lemma}
\begin{proof} First define by induction $(n_i)_{i\in \N}$ such that $n_0=k$ and such that for every $i\geq 0$, $n_{i+1}=\inf\{n\in \N: n>n_i, \X_{2^{n}}\geq \X_{2^{n_i}}^8\}$. Note that for $i\in \N$, $\X_{2^{n_{2i+2}}}\geq \X^8_{2^{n_{2i+1}}} \geq \X_{2^{n_i+1}}^8,$ so by Lemma \ref{glue}, there exists a sequence $\{I_i\}_{i\in \N}$ such that for every $i\in \N$, $I_i\in \I_{n_{2i}}$ and $I_{i+1}$ is glued on $I_i$. On this event, note that for every $j\in \N$ and $x\in I_j $, $x$ is at distance at least $\sum_{i=0}^{j-1} L_{n_{2i}}$ of $\X_{2^k}$ so $d_{H}(\T_{\X_{2^k}}, \T) \geq \sum_{i=0}^{\infty} L_{n_{2i}}$. Similarly we have $d_{H}(\T_{\X_{2^k}}, \T) \geq \sum_{i=0}^{\infty} L_{n_{2i+1}}$, hence 
\[ d_{H}(\T_{\X_{2^k}}, \T)\geq  \frac{1}{2}\sum_{i=0}^{\infty} L_{n_{i}}.\]

Finally we compare $\sum_{i=0}^{\infty} L_{n_{i}}$ with $\sum_{n=k+1}^{\infty} \frac{\log \X_{2^{n}}}{2^n}$. By definition of $\{n_i\}_{i\in \N}$ we have:
\begin{equation*} \sum_{n=k}^{\infty} \frac{\log \X_{2^{n}}}{2^n}
 = \sum_{i=0}^{\infty} \sum_{n=n_i}^{n_{i+1}-1}\frac{\log \X_{2^{n}}}{2^n} \leq \sum_{i=0}^{\infty} \sum_{n=n_i}^{n_{i+1}-1}\frac{8\log \X_{2^{n_i}}}{2^n} \leq 64 \sum_{i=0}^{\infty} \frac{\log \X_{2^{n_i}}}{2^{n_i+2}} = 64 \sum_{i=0}^{\infty} L_{n_{i}}. \notag 
\end{equation*}
This concludes the proof.
\end{proof}

 The previous lemma proves that when $\sum \frac{\log \X_{2^{n}}}{2^n}=\infty$ the tree is not compact, thus finishing the proof of Theorem \ref{THM1}. The next lemma gives a more precise description of the geometry of the tree in the non-compact case: "the tree is infinite in every direction".
\begin{lemma} \label{NonCompact}
Suppose that $\sum \frac{\log \X_{2^{n}}}{2^n}=\infty$ then almost surely for every $a<b<c$, $[a,b]^{\uparrow c}$ has infinite diameter.
\end{lemma}
\begin{remark} An equivalent result is proved in Le Gall and Le Jan \cite{IntroICRT1} for non-compact L\'evy trees: the set of values taken by the height process on any non-trivial open interval contain a half line $[a,\infty)$.
\end{remark}
\begin{proof} 
First one may adapt the argument of the proof of Lemma \ref{manger} to prove that for every $k\in \N$ large enough and $I\in \I_k$,
\begin{equation} \diam((Y_{a_I},Y_{a_I+1}]^{\uparrow Y_{a_{I}+1}}) \geq  L_{k}+\frac{1}{2} \left (\sum_{i=1}^{\infty} L_{n_{2i+1}}+\sum_{i=1}^{\infty} L_{n_{2i}} \right ) \geq \frac{1}{64} \sum_{n=n_2}^{\infty} \frac{\log \X_{2^{n}}}{2^n}= \infty, \label{18102}\end{equation}
where $\{n_i\}_{i\in \N}$ is defined in the proof of Lemma \ref{manger} and $a_I$ in Lemma \ref{glue}. We leave the details to the reader. 

We now fix $a<b<c\in \R^+$. Since conditionally on $\f:=\sigma(\mu,(Y_i)_{i\geq 1})$, $(Z_i)_{i\geq 1}$ are independent random variables with law $(p_{Y_i})_{i\in \N}$, we have for every $m\in \N$,
\[ \proba \left ( \left . \forall I \in \I_m, Z_{a_I}\notin [a,b] \right | \f  \right )=\prod_{I'\in \I_{m}}  \left (1-\frac{\mu[a,b]}{\mu[0,Y_{a_{I'}}]} \right )  \leq \exp \left (-\frac{\mu[a,b]\# \mathcal I_m}{\mu[0,\X_{2^{m+1}}]} \right). \]
Since $\mu$ has full support it follows from Lemmas \ref{5} and \ref{Long} that the right-hand side above converges to $0$ as $m\to \infty$. Therefore for every $n\in \N$, there exists almost surely $m\geq n$ and $I\in \mathcal I_{m}$ such that $Z_{a_I}\in [a,b]$. It follows from \eqref{18102} that if $m$ is large enough $[a,b]^{\uparrow Y_{a_I}}$ has infinite diameter, hence $[a,b]^{\uparrow c}$ also has infinite diameter. Since $a<b<c$ are arbitrary and since rational numbers are dense on $\R^+$, the desired  claim follows.
\end{proof}
%
%
\section{Fractal dimensions : proof of theorem \ref{THM3}}  \label{Section 7}
In this section we prove Theorem \ref{THM3}. By Lemma \ref{FalconPunch}, it suffices to upper bound the Minkowski dimensions and to lower bound the Packing and Hausdorff dimension. We obtain the upper bounds from some simple cover of $\T$ and we derive the lower bounds from Lemma \ref{Hausdorff}. \subsection{Upper bound for the Minkowski dimensions}
First from the change of variables $u=\X_l$, note that the upper bound for the Minkowski dimensions given by Theorem \ref{THM3} are equivalent to
\[ \text{(a)} \quad \overline{\dim}(\T)\leq \limsup_{l\to \infty} \frac{ \log l\X_l}{\log l} \quad \text{and} \quad \text{(b)} \quad \underline{\dim}(\T)\leq \liminf_{l\to \infty} \frac{ \log l\X_l}{\log l} \ \  \text{when} \ \ \log \X_l=l^{o(1)} .  \]

Then for every $l\in \R$, $\T_{\X_{l}}$ has total length $\X_{l}$, hence one can construct a cover of $\T_{\X_l}$ using $l\X_l$ balls of radius $2/l$. By increasing the radius of those balls by $d_H(\T_{\X_l},\T)$ one obtains a cover of $\T$. So for every $l\in \R^+$, 
\begin{equation} N_{2/l+d_H(\X_l,\X)} \leq l\X_l. \label{1102} \end{equation}
The claims (a) and (b) are applications of the inequality in \eqref{1102}.

Toward proving (a), we may assume that $\log \X_l =O(\log l)$ since otherwise the bound is trivial. It follows from Lemma \ref{BIG LEMMA} that $d_{H}(\T_{\X_{2^{k-1}}}, \T_{\X_{2^{k}}})=O(k/2^k)$ and summing over all $k\geq \log_2(l)$, we obtain $d_H(\X_{l},\T)=O (\log(l)/l )$. Therefore by \eqref{1102},
\[ \overline{\dim}(\T)=\limsup_{l\to \infty} \frac{ \log N_{1/l}}{\log l} =  \limsup_{l\to \infty} \frac{ \log N_{2/l+d_H(\T_{\X_l},\T)}}{-\log \left (2/l+d_H(\T_{\X_l},\T) \right )} \leq \limsup_{l\to \infty} \frac{\log (l\X_l) }{ \log l}.  \]
and (a) follows. (b) can be treated similarly by observing that  Lemma \ref{BIG LEMMA} and $\log \X_l=l^{o(1)}$ implies that $d_H(\X_{l},\T)= l^{-1+o(1)}$. We leave the details to the reader. This concludes the proof.

\subsection{Lower bound for the Packing dimension and the Hausdorff dimension}

In this section we show that almost surely,
\[  \dim_P(\T)\geq  \alpha :=1+\limsup_{l\to \infty} \frac{\log l}{\log \E[\mu[0,l]]} \quad \text{and} \quad \dim_H(\T)\geq \beta :=1+\liminf_{l\to \infty} \frac{\log l}{\log \E[\mu[0,l]]}. \]
To this end, by lemma \ref{Hausdorff} it suffices to prove that if $A$ is a random variable with law $p$ then almost surely for every $\delta>0$, $\liminf p(B(A,\e)) \e^{-\alpha-\delta}<\infty$ and $p(B(A,\e))=O(\e^{\beta+\delta})$ as $\e\to 0$. The two previous inequalities can be proved via an elementary computation using $\mu[0,l]\sim \E[\mu[0,l]]$ (Lemma \ref{5}),
\[ \text{(a)} \quad p(B(A,d(A,\T_{Y_i})))\leq \frac{1}{Y_i^{1+o(1)} \mu[0,Y_i]} \quad \text{ and }\quad \text{(b)} \quad d(A,\T_{Y_{i+1}})\geq \mu[0,Y_i]^{-1+o(1)}. \]
We omit the details and focus on the proof of (a) and (b).

Toward (a), let $\gamma$ be the geodesic path from $0$ to $A$ and let for every $i\in \N$, $j_i:=\min\{j\geq i : (Y_j,Y_{j+1}]\cap \gamma \neq \emptyset\}$. Note that since by Theorem \ref{THM1} almost surely $A\notin \R^+$, 
\[ B(A,d(A,\T_{Y_i})) \subset \{Z_{j_i}\}\cup (Y_{j_i},Y_{j_{i}+1}]^{\uparrow Y_{j_{i}+1}}. \]
Furthermore we have by Lemma \ref{10}, $\mu(Y_{j_i},Y_{j_{i+1}}]\leq \frac{\log^2 Y_i}{Y_i}$. Therefore by Lemma \ref{E=MC2} (iii), conditionally on $(\mu, \{Y_j\}_{j\in \N})$, for every $i$ large enough with probability at least $1-Y_{j_i}^{-5}\geq 1-Y_{i}^{-5}$,
\begin{equation} p\left ( B(A,d(A,\T_{Y_i})) \right ) \leq p\left (]Y_{j_i},Y_{j_{i+1}}]^{\uparrow Y_{j_{i+1}}} \right )\leq 2 \frac{\log^6 Y_{j_i}}{Y_{j_i} \mu[0,Y_{j_i}]} \leq 2 \frac{\log^6 Y_i}{Y_i\mu[0,Y_i]}. \label{NAVIGOA} \end{equation}
Moreover by Lemma  \ref{8}, we have $i=O(Y_i^2)$, hence $\sum_{i=1}^\infty Y_i^{-5}<\infty$. The Borel--Cantelli lemma then yields that almost surely \eqref{NAVIGOA} holds for every $i$ large enough, hence (a) holds.

Toward (b), let us first upper bound $\{p(S_n)\}_{n\in \N}$ where for $n\in \N$, $S_n$ denotes the set of $x\in \T$ such that $d(x,[0,\X_{2^n}])\leq \delta_n:=\frac{1}{2^n n^6}$. Let 
for every $n\in \N$,
\[a_n:=\max\{a : Y_a\leq \X_{n^22^n}\} \quad \text{ and } \quad  S'_n:=\left \{x\in \left (\X_{2^n}+\delta_n, Y_{a_n} \right ] : \left [x-\delta_n,x \right ]\cap \{Y_i\}_{i\in \N}\neq \emptyset \right \}. \]
Note that for every $n\in \N$, $S_n\subset ([0,\X_{2^n}+\delta_n]\cup S'_n)^{\uparrow Y_{a_n}}.$
Therefore by Lemma \ref{E=MC2} (ii), \ref{5} and \ref{10}, almost surely for every $n$ large enough: 
\begin{equation} p(S_n)\leq 2p_{Y_{a_n}} \left ([0,\X_{2^n}+\delta_n]\cup S'_n \right ) = 2\frac{\mu[0,\X_{2^n}+\delta_n]+\mu(S'_n)} {\mu[0,\X_{n^22^n}]-\mu(Y_{a_n},\X_{n^22^n}]}\leq  4\frac{2^n+\mu(S'_n)} {n^22^n}. \label{1210} \end{equation}

Furthermore since conditionally on $\mu$, $(Y_i)_{i\in \N}$ is a Poisson point process of rate $\mu[0,l]dl$, we have by Fubini's theorem, for every $n\in \N$:
\begin{align*}
 \E[ \mu(S'_n) |\mu] 
  & \leq   \int_{0}^{\X_{n^22^n}} \proba \left [ \left .  \left [x-\delta_n,x \right ]\cap\{Y_i\}_{i\in \N} \neq \emptyset \right  |\mu  \right ] d\mu(x) \\
 & \leq   \int_{0}^{\X_{n^22^n}} 1-e^{-\delta_n\mu \left[0, \X_{n^22^n} \right ]}d\mu(x) \\
 & \leq  \delta_n\mu \left[0, \X_{n^22^n} \right ]^2.
 \end{align*}
 It directly follows from Lemma \ref{5} that almost surely $\E[ \mu(S'_n) |\mu]=O(2^n/n^2)$ as $n\to \infty$. Thus by Markov's inequality and the Borel--Cantelli lemma almost surely $\mu(S'_n)=O(2^n)$. Therefore by \eqref{1210}, $p(S_n)=O\left (1/n^2 \right )$, hence by the Borel--Cantelli lemma almost surely for every $n$ large enough, $A\notin S_n$. 
 
Finally let for every $i\in \N$, $n_i:=\inf\{n\in \N, Y_{i+1}\leq \X_{2^{n}}\}$. We have by Lemma \ref{5} and \ref{10} almost surely $\mu[0,Y_i] \geq 2^{n_i+O(1)}$. Hence, since for every $i$ large enough $A\notin S_{n_i}$, we have,
\begin{equation*} d(A,Y_{i+1})\geq d(A,\X_{2^{n_i}}) \geq \frac{1}{2^{n_i} {n_i}^6} \geq \mu[0,Y_i]^{-1+o(1)}. \label{NAVIGOB} \end{equation*}
This concludes the proof of $(b)$ and therefore of Theorem \ref{THM3}.
\paragraph{Acknowledgment}
Thanks are due to Nicolas Broutin for interesting conversations and numerous advice on earlier versions of this paper.

\bibliographystyle{unsrt}

\appendix
\section{Appendix}
First let us prove an exponential concentration inequality for general P\'olya urns.
\begin{lemma} \label{P\'olya} Let $\{m_n\}_{n\geq 0}$ be a positive real-valued sequence. Let $(A_n)_{n\geq 0}$ be a sequence of positive real-valued random variables such that $A_0\leq m_0$ and such that for every $n\geq 0$, 
\[ \proba \left ( \left . A_{n+1}=A_n+m_{n+1} \right | A_n \right)=\frac{A_n}{M_n} \quad ; \quad \proba \left ( \left . A_{n+1}=A_n \right | A_n \right )=\frac{M_n-A_n}{M_n}, \]
where for every $n\geq 0$, $M_n=\sum_{i=0}^n m_n$. We say that in this case $(A_n)_{n\geq 0}$ is a $(A_0,\{m_i\}_{i\geq 0} )$ P\'olya urn. 

\begin{compactitem}
\item[a)] If $\sum_{m=0}^\infty \frac{m_n^2}{M_n^2}<\infty$, then almost surely for every $a\geq 0$ and $t\in \R^+$,
\[ \proba \left ( \left . \sup_{i\geq a} \left | \frac{A_i}{M_i} - \frac{A_a}{M_a} \right  | >   t\frac{A_a}{M_a}  \right | A_a \right ) \leq 2 \exp \left (-\frac{\frac{t^2}{4}\frac{A_a}{M_a}}{\sum_{n> a} \frac{m_n^2}{M_n^2}+t\max \left (\sum_{n> a} \frac{m_n^2}{M_n^2},  \max_{n>a} \frac{m_n}{M_n} \right )}\right ). \]
\item[b)]  If $\{m_n\}_{n\in \N}$ is bounded, then almost surely for every $a\geq 0$ and $t\in \R^+$,
\[ \proba \left ( \left . \sup_{i\geq a} \left | \frac{A_i}{M_i} - \frac{A_a}{M_a} \right  | >   t\frac{A_a}{M_a}  \right | A_a \right ) \leq 2 \exp \left (-\frac{t^2}{4(1+t)}\frac{A_a}{\max_{n>a} m_n} \right ). \]
\end{compactitem}
\end{lemma}
\begin{remark} Note that Lemma \ref{P\'olya} implies that almost surely $\{\frac{A_i}{M_i}\}_{i\in \N}$ is a Cauchy sequence and so converges. The statement should then be seen as an estimate on the speed of convergence. 
\end{remark}
\begin{proof} First let us explain why (b) follows from (a). We have for every $a\in \N$,
\[ \sum_{n>a}\frac{m_n^2}{M_n^2} \leq \max_{n>a} m_n \sum_{n>a}\frac{m_n}{M_n^2} \leq \max_{n>a} m_n \int_{M_a}^{+\infty} \frac{dx}{x^2} =\frac{\max_{n>a} m_n}{M_a}, \]
and (b) follows by replacing  $\max_{n>a} \frac{m_n}{M_n}$ and $\sum_{n>a}\frac{m_n^2}{M_n^2}$ by the upper bound $\frac{\max_{n>a} m_n}{M_a}$ in (a).

We focus henceforth on (a). To simplify the notation set for every $n\in \N$, $X_n:=\frac{A_n}{M_n}$ and $\delta_n:=\frac{m_n}{M_n}$. Also we write for every $a\in \N$, $\E^{(a)}[\dots]=\E[\dots |X_a]$.  We first prove by induction that for every $a,b,c\in \N$ and $\lambda\in \R$ satitisfying
\begin{equation} a\leq b\leq c\quad; \quad |\lambda| \leq \Lambda_a:=\frac{1}{4\max \left \{\sum_{n> a} \delta_n^2, \max_{n>a} \delta_n \right\}}\label{1701} \end{equation}
 we have
\begin{equation}  \E^{(a)} \left [e^{\lambda X_c} \right ] \leq f(a,b,c,\lambda):= \E^{(a)} \left [e^{\lambda \left(1+\lambda\sum_{n=b+1}^{c}\delta_n^2\right )X_b}\right ]. \label{lol} \tag{$P(a,b,c,\lambda)$}\end{equation}
Note that when $b=c$, $P(a,b,c,\lambda)$ is trivial. Therefore it suffices to prove that for every $a,b,c, \lambda$ such that $a\leq b< c$ and $\lambda\leq \Lambda_a$ that $f(a,b,c,\lambda)\leq f(a,b+1,c,\lambda)$. Fix $a\leq b<c$, $\lambda\leq \Lambda_a$  and let $\gamma:= \lambda  \left(1+\lambda\sum_{n=b+2}^{c}\delta_n^2\right )$. We have,
\begin{align} f(a,b+1,c,\lambda) & =  \E^{(a)}\left [ e^{\gamma X_{b+1}} \right ] \notag \\
 & =   \E^{(a)}\left [ e^{\gamma X_b}\E^{(b)} \left [ e^{\gamma \left ( X_{b+1}-X_b \right)} \right ] \right ] \notag{} \\ & =  
 \E^{(a)}\left [ e^{\gamma X_b} \left (X_b e^{\gamma \left ( \frac{A_b+m_{b+1}}{M_{b+1}} -X_b\right )}+\left (1- X_b\right ) e^{ \gamma \left (\frac{A_b}{M_{b+1}}- X_b \right )} \right )\right ] \notag{}
\\ & =   \E^{(a)}\left [ e^{\gamma  X_b} \left ( X_b e^{\gamma \delta_{b+1}\left (1- X_b\right)}+\left (1- X_b\right ) e^{-\gamma \delta_{b+1} X_b} \right )\right ]. \label{1401}
 \end{align}
 Furthermore by \eqref{1701}, $| \gamma| \leq \frac{5}{4}|\lambda| \leq \frac{5}{16\delta_{b+1}}$, hence since for every $0\leq x\leq 1$ and $|c| \leq \frac{5}{16}$, $xe^{c(1-x)}+(1-x)e^{-cx}\leq e^{\frac{16}{25}c^2x}$ , we have
\begin{equation} X_b e^{\gamma \delta_{b+1}\left (1- X_b\right)}+\left (1- X_b\right ) e^{-\gamma \delta_{b+1} X_b} \leq e^{\frac{16}{25}\gamma^2\delta_{b+1}^2 X_b}. \label{14012} \end{equation}
Finally by \eqref{1401}, \eqref{14012}, and $|\gamma|\leq \frac{5}{4}|\lambda|$,
\begin{equation*}f(a,b+1,c,\lambda) \leq   \E^{(a)}\left [ e^{\left (\gamma +\frac{16}{25}\gamma^2\delta_{b+1}^2 \right ) X_b} \right ] \leq \E^{(a)}\left [ e^{\left (\gamma +\lambda^2\delta_{b+1}^2 \right ) X_b} \right ] =f(a,b,c,\lambda) . \label{08103}\end{equation*}
This concludes our proof by induction of $P(a,b,c,\lambda)$. 

We now fix $a\in \N$. For every $n\in \N$ and $0\leq \lambda \leq \Lambda_a$, by $P(a,a,n,\lambda)$ and $P(a,a,n,-\lambda)$, we have the sub-Gaussian bound, $\E^{(a)}\left [e^{\lambda |X_n-X_a|} \right ] \leq  2e^{\lambda^2V_a/2}$ where $V_a:=2X_a\sum_{i>a} \delta_i^2$. 

Furthermore note that $\{X_n\}_{n\geq a}$ is a martingale, and hence that for every $\lambda\in \R^+$, $\{e^{\lambda |X_n-X_a|}\}_{n\geq a}$ is a sub-martingale. It follows by Doob's inequality that for every $t \in \R^+$ and $0\leq \lambda \leq \Lambda_a$,
\begin{equation} \proba^{(a)} \left (\sup_{n\geq a} |X_n-X_a| \geq t \right ) = \proba^{(a)} \left (\sup_{n\geq a} e^{\lambda |X_n-X_a|} \geq e^{\lambda t} \right ) \leq \sup_{n\geq a} \frac{\E^{(a)}\left[e^{\lambda |X_n-X_a|}\right ]}{e^{\lambda t}} \leq 2 e^{\lambda^2\frac{V_a}{2} -\lambda t} \label{10103} \end{equation}
On the one hand, for every $0\leq t\leq V_a \Lambda_a$, taking $\lambda:= t/V_a$ in \eqref{10103} gives,
\begin{equation}\proba^{(a)} \left (\sup_{n\geq a} |X_n-X_a| \geq t \right ) \leq 2e^{-\frac{t^2}{2V_a}}. \label{1902} \end{equation}
On the other hand, for every $t > V_a\Lambda_a$, taking $ \lambda:= \Lambda_a$ in \eqref{10103} gives,
\begin{align*} \proba^{(a)} \left (\sup_{n\geq a} |X_n-X_a| \geq t \right ) \leq 2e^{\Lambda_a^2\frac{V_a}{2}-\Lambda_at} = 
2e^{\frac{t^2}{V_a} \left ( \frac{1}{2}\left( \frac{t}{\Lambda_a V_a} \right )^{-2}- \left( \frac{t}{\Lambda_a V_a} \right )^{-1}\right )},
\end{align*}
hence since for every $x\geq 1$, $\frac{1}{2x^2}-\frac{1}{x}\leq -\frac{1}{2+x}$,
\begin{equation} \proba^{(a)} \left (\sup_{n\geq a} |X_n-X_a| \geq t \right ) \leq 2e^{-\frac{t^2}{2V_a+t/\Lambda_a}}. \label{10102} \end{equation}
By \eqref{1902} the last inequality is also true for $0\leq t\leq V_a \Lambda_a$. 
The desired inequality then directly follows from a reorganization of the different terms in \eqref{10102}. 
We omit the straightforward details.
\end{proof}
The following lemma is a version of the strong law of large number.
\begin{lemma} \label{strong} Let $p>0$, $\{X_i\}_{i\in \N}$ be a family of independent Bernoulli random variables with mean $p$. Let $\{a_i\}_{i\in \N}$ be a positive real-valued sequence and let for every $n\in \N$, $A_n:=\sum_{i=1}^n a_i$. Suppose that $A_n\limit \infty$ and $\sum_{n=1}^{\infty} \frac{a_n^2}{ A_n^2}<\infty$, then almost surely  \[ \sum_{i=1}^{n} a_iX_i\sim pA_n.\]
\end{lemma}
\begin{proof} Since $\sum_{n=1}^{\infty} \frac{a_n^2}{ A_n^2}<\infty$ and $\{X_i\}_{i\in \N}$ are independent random variables, the classical three series theorem implies that
$S_n:=\sum_{i=1}^n \frac{a_i (X_i-p)}{A_i}$ almost surely converges as $n\to \infty$. Therefore,
\[ \sum_{i=1}^n \frac{a_i (X_i-p)}{A_n}=\sum_{i=1}^n \frac{A_i}{A_n}(S_i-S_{i-1})=S_n-\sum_{i=1}^{n-1}S_i \frac{A_{i+1}-A_{i}}{A_n} \limit_{n\to \infty} 0. \qedhere\]
\end{proof}
\begin{remark} If there exists $\e>0$ such that $a_i=O(A_i^{1-\e})$, then as $n$ goes to infinity, 
\[ \sum_{n=1}^{n} \frac{a_i^2}{A_i^2} =O\left (\sum_{i=1}^{n} \frac{a_i}{A_i^{1+\e}}  \right ) = O \left (\int_1^{A_n} \frac{dx}{x^{1+\e}} \right ) =O(1).\]
\end{remark}
\end{document}